\documentclass[11pt,a4paper]{amsart}
\usepackage{amssymb, amsthm, enumerate, amsmath, mathrsfs}

\usepackage{tikz}
\usetikzlibrary{calc,patterns,through}

\usepackage{listings}
\lstnewenvironment{code}{\lstset{language=python,basicstyle=\ttfamily,commentstyle=\ttfamily,tabsize=2,breaklines=true}}{}

\theoremstyle{plain}
\newtheorem{thm}{Theorem}[section]
\newtheorem{theorem}[thm]{Theorem}
\newtheorem{lemma}[thm]{Lemma}
\newtheorem{cor}[thm]{Corollary}
\newtheorem{corollary}[thm]{Corollary}
\newtheorem{prop}[thm]{Proposition}

\newtheorem{problem}[thm]{Problem}
\newtheorem{claim}{}[thm] 
\newenvironment{subproof}{\begin{proof}[Subproof.]}{\end{proof}}
\usepackage[colorlinks,linkcolor=blue,anchorcolor=black,citecolor=blue,filecolor=blue,menucolor=blue,runcolor=blue,urlcolor=blue]{hyperref}

\DeclareMathOperator{\cl}{cl}
\DeclareMathOperator{\fcl}{fcl}
\DeclareMathOperator{\si}{si}
\DeclareMathOperator{\co}{co}
\DeclareMathOperator{\gf}{GF}

\newcommand{\del}{\hspace{-0.5pt}\backslash}
\newcommand{\delete}{\del}
\newcommand{\contract}{/}
\newcommand{\nfd}{\ensuremath{(F_{7}^{-})^{*}}}
\newcommand{\agde}{\ensuremath{\mathrm{AG}(2, 3)\hspace{-0.5pt}\backslash e}}

\newcommand{\dy}{\ensuremath{\Delta}-\ensuremath{\nabla}}
\newcommand{\hydra}{\mathbb{H}}

\begin{document}
 \title[The $\{U_{2,5}, U_{3,5}\}$-fragile matroids]{The structure of $\{U_{2,5}, U_{3,5}\}$-fragile matroids}

\author{Ben Clark}
\address{Department of Mathematics, Louisiana State University, Baton Rouge, Louisiana, USA}
\email{bclark@lsu.edu}
\author{Dillon Mayhew}
\address{School of Mathematics and Statistics, Victoria University of Wellington, New Zealand}
\email{dillon.mayhew@vuw.ac.nz}
\author{Stefan van Zwam}
\address{Department of Mathematics, Louisiana State University, Baton Rouge, Louisiana, USA}
\email{svanzwam@math.lsu.edu}
\author{Geoff Whittle}
\address{School of Mathematics and Statistics, Victoria University of Wellington, New Zealand}
\email{geoff.whittle@vuw.ac.nz}
\date{\today}


\begin{abstract}
  Let $\mathcal{N}$ be a set of matroids. A matroid $M$ is \emph{strictly $\mathcal{N}$-fragile} if $M$ has a member of $\mathcal{N}$ as minor and, for all $e \in E(M)$, at least one of $M\delete e$ and $M\contract e$ has no minor in $\mathcal{N}$. In this paper we give a structural description of the strictly $\{U_{2,5},U_{3,5}\}$-fragile matroids that have six inequivalent representations over $\mathrm{GF}(5)$. Roughly speaking, these matroids fall into two classes. The matroids without an $\{X_8, Y_8, Y_8^{*}\}$-minor are constructed, up to duality, from one of two matroids by gluing wheels onto specified triangles. On the other hand, those matroids with an $\{X_8, Y_8, Y_8^{*}\}$-minor can be constructed from a matroid in $\{X_8, Y_8, Y_8^{*}\}$ by repeated application of elementary operations, and are shown to have path width 3.  
  
  The characterization presented here will be crucial in finding the explicit list of excluded minors for two classes of matroids: the \emph{Hydra-5-representable matroids} and the \emph{2-regular matroids}.
\end{abstract}

\maketitle

\section{Introduction}

Certain minor-closed classes of matroids can be characterized by the property of having a representation by a matrix over a certain \emph{partial field} (see Section \ref{preliminaries} for a definition). In this paper we consider two partial fields, the \emph{Hydra-5 partial field} $\hydra_5$ introduced by Pendavingh and Van Zwam \cite{pendavingh2010confinement} and the \emph{2-regular partial field} $\mathbb{U}_2$ introduced by Semple \cite{Semple} (called the \emph{2-uniform partial field} by Pendavingh and Van Zwam). This paper is a major step toward solving the following problem:
\begin{problem}
  Find the full sets of excluded minors for the classes of $\hydra_5$-representable and $\mathbb{U}_2$-representable matroids.
\end{problem}
These classes are closely related (the latter is a subset of the former), and it seems that any differences in their analysis can be confined to finite case checks. The excluded minors for the class of $\hydra_5$-representable matroids are the first step in a 5-step process to find the excluded minors for the class of $\gf(5)$-representable matroids, as described in \cite{mayhew2010stability}. The excluded minors for the class of $\mathbb{U}_2$-representable matroids are a first step in a process to find the excluded minors for the Golden Mean matroids (see \cite{Whi05}), and should shed light on the following problem:
\begin{problem}
  Find the partial field $\mathbb{P}$ such that the set of $\mathbb{P}$-representable matroids is exactly the set of matroids representable over $\gf(4)$ and all larger fields.
\end{problem}
Van Zwam \cite[Conjecture 4.4.2]{van2009partial} conjectured a possible answer.

It is well-established that an important step in the search for the excluded minors mentioned above, is to find the matroids that are \emph{fragile} with respect to an appropriately chosen set of matroids in the class. In particular, in forthcoming work, Clark, Oxley, Semple, and Whittle will prove the following.
\begin{theorem}\label{thm:fragileexcludedconnection}
  Let $M$ be a sufficiently large excluded minor for the class of matroids representable over $\mathbb{H}_5$ or $\mathbb{U}_2$. Then $M$ has a $\{U_{2,5}, U_{3,5}\}$-minor, and if $M$ has a pair of elements $a,b$ such that $M\del a,b$ is $3$-connected with a $\{U_{2,5}, U_{3,5}\}$-minor, then $M\del a,b$ is a $\{U_{2,5}, U_{3,5}\}$-fragile matroid. 
\end{theorem}

In this paper we give a structural characterization of these $\{U_{2,5},U_{3,5}\}$-fragile matroids. We need some definitions before stating our main result.

Let $\mathcal{N}$ be a set of matroids. We say that a matroid $M$ has an \textit{$\mathcal{N}$-minor} if some minor $N$ of $M$ is isomorphic to a member of $\mathcal{N}$. Let $x \in E(M)$. If $M\del x$ has an $\mathcal{N}$-minor, then $x$ is $\mathcal{N}$-\textit{deletable}. If $M/x$ has an $\mathcal{N}$-minor, then $x$ is $\mathcal{N}$-\textit{contractible}. If $x$ is both $\mathcal{N}$-deletable and $\mathcal{N}$-contractible, then we say that $x$ is \textit{$\mathcal{N}$-flexible}, and if $x$ is neither, it is \emph{$\mathcal{N}$-essential}. We say that the matroid $M$ is \textit{$\mathcal{N}$-fragile} if no element of $M$ is $\mathcal{N}$-flexible. If, in addition, $M$ has an $\mathcal{N}$-minor, then we say that $M$ is \textit{strictly $\mathcal{N}$-fragile}. In this paper, whenever we omit the ``$\mathcal{N}$-'' prefix from these terms, assume that $\mathcal{N}=\{U_{2,5}, U_{3,5}\}$.  

A matroid has \emph{path width} at most $k$ if there is an ordering $(e_1, e_2, \ldots, e_n)$ of its ground set such that $\{e_1, \ldots, e_t\}$ is $k$-separating for all $t$. \emph{Gluing a wheel to $M$} is the process of taking the generalized parallel connection of $M$ with $M(\mathcal{W}_n)$ along a triangle $T$, and deleting any subset of $T$ containing the rim element. In this paper we prove the following:
\begin{thm}
  Let $\mathbb{P} \in \{\mathbb{H}_5, \mathbb{U}_2\}$. Let $M$ be a strictly fragile, $\mathbb{P}$-representable matroid with $|E(M)| \geq 10$. Then either
  \begin{enumerate}
    \item $M$ or $M^*$ can be obtained by gluing up to three wheels to $U_{2,5}$; or
    \item $M$ has path width at most 3.
  \end{enumerate}
\end{thm}
In fact, our main result (Theorem \ref{thm:mainresult}) describes the structure of the matroids in this class in much more detail, using the concept of \emph{generalized $\Delta-Y$ exchange}.

The paper is organized as follows. In Section \ref{preliminaries}, we collect basic results on the generalized $\Delta-Y$ exchange and partial fields. In Section \ref{sec:noX} we describe the matroids in our classes on at most 9 elements. We identify three special matroids $X_8, Y_8, Y_8^*$, and describe the exact structure of all matroids in our classes that have no $\{X_8, Y_8, Y_8^*\}$-minor. These results rely on computer calculations carried out with SageMath \cite{sage,sage-matroid} and a result from \cite{chun2013fan}. In Section \ref{sec:X} we describe the exact structure of the matroids in our class that do have a $\{X_8, Y_8, Y_8^*\}$-minor, and state our main result. The remainder of the paper is devoted to the proof of the main result.

Any undefined matroid terminology will follow Oxley \cite{oxley2011matroid}. In addition, we say $x$ is in the \emph{guts} of a $k$-separation $(A,B)$ if $x \in \cl(A-x) \cap \cl(B-x)$, and we say $x$ is in the \emph{coguts} of $(A,B)$ if $x$ is in the guts of $(A,B)$ in $M^*$. We also use the shorthand $x\in \cl^{(*)}(X)$ if either $x\in \cl(X)$ or $x\in \cl^{*}(X)$. 

\section{Preliminaries}
\label{preliminaries}

\subsection{Partial fields}
We start with a brief introduction to partial fields. We refer to \cite{pendavingh2010confinement} for a more thorough treatment and proofs.

Let $R$ be a ring and $G$ a subgroup of the units of $R$ with $-1 \in G$. We say $\mathbb{P} = (R,G)$ is a \emph{partial field}. A matrix over $\mathbb{P}$ is a \emph{$\mathbb{P}$-matrix} if $\det(D) \in G \cup \{0\}$ for all square submatrices $D$ of $A$. If $\phi:R\to S$ is a ring homomorphism, then $\phi(A)$ is an $(S, \phi(G))$-matrix. 

A rank-$r$ matroid with ground set $E$ is \emph{$\mathbb{P}$-representable} if there is an $r\times E$ $\mathbb{P}$-matrix such that for each $r\times r$ submatrix $D$, $\det(D) \neq 0$ if and only if the corresponding subset of $E$ is a basis of $M$. We write $M = M[A]$. Note that $M[A] = M[\phi(A)]$.

Denote the group generated by a set $X$ by $\langle X \rangle$. The partial fields relevant to this paper are
\begin{itemize}
  \item The \emph{near-regular partial field}
  \begin{align*}
    \mathbb{U}_1 = \left(\mathbb{Z}[\alpha,\frac{1}{\alpha},\frac{1}{1-\alpha}], \{\pm \alpha^i (1-\alpha)^j : i,j \in \mathbb{Z}\}\right);
  \end{align*}
  \item The \emph{2-regular partial field} 
  \begin{align*}
  \mathbb{U}_2 = \left(\mathbb{Q}(\alpha,\beta), \langle -1, \alpha, \beta, 1-\alpha, 1-\beta,\alpha-\beta\rangle\right);  
  \end{align*}
  \item The \emph{Hydra-5 partial field}
  \begin{align*}
    \mathbb{H}_5 = \big(\mathbb{Q}(\alpha, \beta, \gamma), \langle -1, & \alpha, \beta, \gamma, \alpha-1,\beta-1,\gamma-1,\alpha-\gamma,\\ &\gamma-\alpha\beta,(1-\gamma)-(1-\alpha)\beta\rangle\big).
  \end{align*}
\end{itemize}
Note that instead of the field of fractions $\mathbb{Q}(\alpha,\beta,\gamma)$ we could have taken the smallest subring in which all group elements are invertible, as was done for $\mathbb{U}_1$. We have

\begin{lemma}\label{lem:H5sixhoms}
  \cite[Lemma 5.17]{pendavingh2010confinement}
  Let $M$ be a 3-connected matroid with a $\{U_{2,5},U_{3,5}\}$-minor. Then $M$ is $\mathbb{H}_5$-representable if and only if $M$ has six inequivalent representations over $\gf(5)$.
\end{lemma}

By combining partial-field homomorphism results of \cite[Theorem 1.3(ii)]{pendavingh2010confinement} and \cite[Corollary 3.1.3]{Semple} we have the following. 
\begin{lemma}
\label{H5hom}
Let $M$ be $\mathbb{P}$-representable for some $\mathbb{P}\in \{\mathbb{U}_2, \mathbb{H}_5\}$. If $\mathbb{F}$ is a field with at least $5$ elements, then $M$ is $\mathbb{F}$-representable.
\end{lemma}

\subsection{Connectivity}
If $M$ is a connected matroid such that $\min\{r(X),r(Y)\}=1$ or $\min\{r^{*}(X),r^{*}(Y)\}=1$ for every $2$-separation $(X,Y)$ of $M$, then we say that $M$ is $3$-connected \textit{up to series and parallel classes}. A subset $S$ of $E(M)$ is a \textit{segment} if every $3$-element subset of $S$ is a triangle. A \textit{cosegment} is a segment of $M^{*}$. 

The next result implies that every $\{U_{2,5}, U_{3,5}\}$-fragile matroid is $3$-connected up to series and parallel classes. 

\begin{prop}
\label{fcon}
\cite[Proposition 4.3]{mayhew2010stability}
Let $M$ be a matroid with a $2$-separation $(A,B)$, and let $N$ be a $3$-connected minor of $M$. Assume $|E(N)\cap A|\geq |E(N)\cap B|$. Then $|E(N)\cap B|\leq 1$. Moreover, unless $B$ is a parallel or series class, there is an element $x\in B$ such that both $M\del x$ and $M/x$ have a minor isomorphic to $N$. 
\end{prop}

The following is an easy consequence of the property that strictly $\{U_{2,5}, U_{3,5}\}$-fragile matroids are $3$-connected up to parallel and series classes. 
\begin{lemma}
\label{incotri}
Let $M$ be a strictly $\{U_{2,5}, U_{3,5}\}$-fragile matroid with at least $8$ elements. If $S$ is a triangle or $4$-element segment of $M$ such that $E(M)-S$ is not a series or parallel class of $M$, then $S$ is coindependent in $M$. If $C$ is a triad or $4$-element cosegment of $M$ such that $E(M)-S$ is not a series or parallel class of $M$, then $C$ is independent. 
\end{lemma}

We also use the following consequence of orthogonality.

\begin{lemma}
\label{clandco}
 Let $(X,\{e\},Y)$ be a partition of $E(M)$. Then $e\in \cl(X)$ if and only if $e\notin \cl^{*}(Y)$.
\end{lemma}

\subsection{Delta-Y exchange}
The generalized $\Delta-Y$ exchange of Oxley, Semple, and Vertigan \cite{oxley2000generalized} is used frequently. Let $M$ be a matroid with $U_{2,k}$-restriction $A$, and suppose $A$ is coindependent in $M$. The \textit{generalized $\Delta-Y$ exchange on $A$}, denoted by $\Delta_A(M)$, is defined to be the matroid obtained from $P_A(\Theta_k,M)\del A$ by relabeling the elements in $E(\Theta_k) - A$ by $A$ using their natural bijection (see Oxley \cite[Proposition 11.5.1]{oxley2011matroid} for the definition of $\Theta_k$). So $E(\Delta_A(M)) = E(M)$. Let $M$ be a matroid such that $M^{*}$ has $U_{2,k}$-restriction $A$, and suppose $A$ is independent in $M$. The \textit{generalized $Y$-$\Delta$-exchange on $A$}, denoted by $\nabla_A(M)$, is defined to be the matroid obtained from $(P_A(\Theta_k,M^{*})\del A)^{*}$ by relabeling the elements in $E(\Theta_k) - A$ by $A$. That is, $\nabla_A(M)=(\Delta_A(M^{*}))^{*}$. 

We now state some of the key properties here.

\begin{lemma}
\label{2.5}
\cite[Lemma 2.5]{oxley2000generalized}
For all $k\geq 2$, the restriction of $(P_A(\Theta_k,M)\del A)^{*}$ to $E(\Theta_k)-A$ is isomorphic to $U_{2,k}$ if and only if $A$ is coindependent in $M$.
\end{lemma}

\begin{lemma}
\label{2.11}
\cite[Lemma 2.11]{oxley2000generalized}
Let $A$ be a coindependent segment in $M$. Then $\nabla_A(\Delta_A(M))$ is well-defined and $\nabla_A(\Delta_A(M))=M$.
\end{lemma}

A key property of the Delta-Y exchange is that it preserves $\mathbb{P}$-representability for any partial field $\mathbb{P}$.

\begin{lemma}
\cite[Lemma 3.4]{oxley2000generalized}
 $\Theta_3$ is regular, and $\Theta_4$ is near-regular.
\end{lemma}

\begin{lemma}
 \label{3.5}
 \cite[Lemma 3.5]{oxley2000generalized}
 Let $k\geq 2$, and let $M$ be a matroid such that $M|A\cong U_{2,k}$. Let $\mathbb{P}$ be a partial field. If $M$ and $\Theta_k$ are $\mathbb{P}$-representable, then the generalized parallel connection $P_A(\Theta_k, M)$ of $\Theta_k$ and $M$ across $A$ is $\mathbb{P}$-representable. 
\end{lemma}

\begin{cor}
\label{3.7}
\cite[Lemma 3.7]{oxley2000generalized}
Let $\mathbb{P}$ be a partial field. Then $M$ is $\mathbb{P}$-representable if and only if $\Delta_A(M)$ is $\mathbb{P}$-representable.
\end{cor}

The following two results on minors are crucial.

\begin{lemma}
\label{2.13}
\cite[Lemma 2.13]{oxley2000generalized}
Suppose that $\Delta_A(M)$ is defined. If $x\in A$ and $|A|\geq 3$, then $\Delta_{A-x}(M\del x)$ is also defined, and $\Delta_A(M)/x=\Delta_{A-x}(M\del x)$.
\end{lemma}

\begin{lemma}
\label{2.16}
\cite[Lemma 2.16]{oxley2000generalized}
Suppose that $\Delta_A(M)$ is defined.
\begin{enumerate}
\item[(i)] If $x\in E(M)-A$ and $A$ is coindependent in $M\del x$, then $\Delta_A(M\del x)$ is defined and $\Delta_A(M)\del x=\Delta_A(M\del x)$;
\item[(ii)] If $x\in E(M)-\cl(A)$, then $\Delta_A(M/x)$ is defined and $\Delta_A(M)/x=\Delta_A(M/x)$;
\end{enumerate}
\end{lemma}  

The next two results are employed for rearranging path sequences.

\begin{lemma}
\label{2.18}
\cite[Lemma 2.18]{oxley2000generalized}
Let $M$ be a matroid, and $S$ and $T$ be disjoint subsets of $E(M)$ such that $|S|\geq 2$ and $|T|\geq 2$. If $S$ and $T$ are both coindependent segments, then 
\begin{equation*}
\Delta_{S}(\Delta_{T}(M))= \Delta_{T}(\Delta_{S}(M)).
\end{equation*}
\end{lemma}

\begin{cor}
\label{2.19}
 \cite[Corollary 2.19]{oxley2000generalized}
 Let $M$ be a matroid, and $S$ and $T$ be disjoint subsets of $E(M)$ such that $|S|\geq 2$ and $|T|\geq 2$. 
 \begin{enumerate}
  \item[(i)] If $S$ and $T$ are both independent cosegments, then 
  \begin{equation*}
\nabla_{S}(\nabla_{T}(M))= \nabla_{T}(\nabla_{S}(M)).
\end{equation*}
 \item[(ii)] If $S$ is an independent cosegment and $T$ is a coindependent segment, then
   \begin{equation*}
\nabla_{S}(\Delta_{T}(M))= \Delta_{T}(\nabla_{S}(M)).
\end{equation*}
 \end{enumerate}
\end{cor}

Let $M$ be a matroid with an $\mathcal{N}$-minor. We say a segment $S$ of a matroid $M$ is \textit{$\mathcal{N}$-allowable} if $S$ is coindependent and some element of $S$ is not $\mathcal{N}$-deletable. Write $\mathcal{N}^* := \{N^{*} : N\in \mathcal{N}\}$. A cosegment $C$ of $M$ is \textit{$\mathcal{N}$-allowable} if the segment $C$ of $M^{*}$ is $\mathcal{N}^{*}$-allowable. We say $S$ is an \textit{$\mathcal{N}$-allowable set} if $S$ is either an $\mathcal{N}$-allowable segment or an $\mathcal{N}$-allowable cosegment of $M$. The next two results give conditions for allowable segments and cosegments to go up the minor order, and are used later in reductions.

\begin{lemma}
\label{asegsgoup}
Let $M$, $M\del x$ be strictly $\mathcal{N}$-fragile matroids such that $M\del x$ has no $\mathcal{N}$-essential elements. If $A\subseteq E(M\del x)$ is an $\mathcal{N}$-allowable segment of $M\del x$, then $A$ is an $\mathcal{N}$-allowable segment of $M$.
\end{lemma}

We omit the straightforward proof.


\begin{lemma}
\label{whencsegsgoup}
Let $M$, $M\del x$ be strictly $\mathcal{N}$-fragile matroids such that $M\del x$ has no $\mathcal{N}$-essential elements. If $A\subseteq E(M\del x)$ is an $\mathcal{N}$-allowable cosegment of $M\del x$ and $x\in \cl(E(M)-(A\cup x))$, then $A$ is an $\mathcal{N}$-allowable cosegment of $M$.
\end{lemma}

\begin{proof}
Assume that $A\subseteq E(M\del x)$ is an $\mathcal{N}$-allowable cosegment of $M\del x$ and $x\in \cl(E(M)-(A\cup x))$. Then $A$ is a coindependent segment of $M^{*}/x$, and $x\notin \cl^{*}(A)$, so $A$ is a coindependent segment of $M^{*}$. Finally, since $M\del x$ has no $\mathcal{N}$-essential elements $A$ has an $\mathcal{N}$-deletable element in $M\del x$ and hence in $M$.    
\end{proof}

We have the following properties on connectivity, rank and closure for the generalized $\Delta-Y$ exchange.

\begin{lemma}
\label{9.3}
 \cite[Lemma 9.3]{hall2004fork}
Let $M$ be a $3$-connected matroid.
\begin{itemize}
 \item[(i)] If $A$ is a segment of $M$, then, for all $A'\subseteq A$, $P_A(\Theta_k, M)\del A'$ is $3$-connected up to series pairs.
 \item[(ii)] If $A$ is a cosegment of $M$, then, for all $A'\subseteq A$, $(P_A(\Theta_k, M)\del A')^{*}$ is $3$-connected up to parallel pairs.
\end{itemize}
\end{lemma}

Performing generalized $\Delta-Y$ exchanges on a set $A$ of a matroid preserves the connectivity of separations that are nested with $A$. We omit the easy proof.

\begin{lemma}
\label{pathDY1}
Let $M$ be a matroid, and let $(X,Y)$ be a partition of $E(M)$. If $A$ is an allowable segment of $M$, and $A\subseteq X$, then $\lambda_M(X)=\lambda_{\Delta_A(M)}(X)$. Moreover, for any $e\notin A$ and $Z\in \{X,Y\}$, $e\in \cl_M(Z)$ if and only if $e\in \cl_{\Delta_A(M)}(Z)$ and $e\in \cl^{*}_M(Z)$ if and only if $e\in \cl^{*}_{\Delta_A(M)}(Z)$.
\end{lemma}

We now work towards a proof that the generalized $\Delta-Y$ operations preserve fragility. We use the excluded-minor characterisation of near-regular matroids.

\begin{thm}
\cite[Theorem 1.2]{hall2011excluded}
\label{thmnr}
The excluded minors for the class of near-regular matroids are
$U_{2,5}$, $U_{3,5}$, $F_{7}$, $F_{7}^{*}$, $F_{7}^{-}$, \nfd, \agde,
$(\agde)^{*}$, $\Delta_{T}(\agde)$, and $P_{8}$.
\end{thm}

We let $\mathcal{EX}(\mathbb{U}_1)$ denote the set of excluded minors for the class of near-regular matroids. We now obtain a representability certificate of losing the $\{U_{2,5}, U_{3,5}\}$-minor.

\begin{cor}
\label{nreg}
Let $M$ be a $\mathbb{U}_2$- or $\mathbb{H}_5$-representable matroid. Then $M$ is near-regular if and only if $M$ has no $\{U_{2,5}, U_{3,5}\}$-minor.
\end{cor}

\begin{proof}
If $M$ is near-regular, then $M$ has no $\{U_{2,5}, U_{3,5}\}$-minor by Theorem \ref{thmnr}. Conversely, suppose $M$ is $\mathbb{P}$-representable for some $\mathbb{P}\in \{\mathbb{U}_2, \mathbb{H}_5\}$ and that $M$ has no $\{U_{2,5}, U_{3,5}\}$-minor. It follows from Theorem \ref{thmnr} and well-known results (see \cite{hall2011excluded, oxley2011matroid}) that, for each matroid $M'$ in $\mathcal{EX}(\mathbb{U}_1)-\{U_{2,5}, U_{3,5}\}$, there is some prime power $q\geq 5$ such that $M'$ is not $\gf(q)$-representable. Therefore $M$ also has no minor in $\mathcal{EX}(\mathbb{U}_1)-\{U_{2,5}, U_{3,5}\}$ by Lemma \ref{H5hom}, so $M$ is near-regular.  
\end{proof}

We can now show how the generalized $\Delta-Y$ operations can be used to build new fragile matroids.

\begin{lemma}
\label{amove}
Let $\mathbb{P}\in \{\mathbb{U}_2, \mathbb{H}_5\}$, and let $M$ be a strictly $\{U_{2,5}, U_{3,5}\}$-fragile $\mathbb{P}$-representable matroid. If $A$ is an allowable segment of $M$ with $|A| \in \{3,4\}$, then $\Delta_A(M)$ is a $\{U_{2,5}, U_{3,5}\}$-fragile $\mathbb{P}$-representable matroid. Moreover, $A$ is an allowable cosegment of $\Delta_A(M)$.
\end{lemma}

\begin{proof}
It follows immediately from Corollary \ref{3.7} and Corollary \ref{nreg} that $\Delta_A(M)$ is a $\mathbb{P}$-representable matroid with a $\{U_{2,5}, U_{3,5}\}$-minor. Moreover, $A$ is an independent cosegment of $\Delta_A(M)$ by Lemma \ref{2.5}. It remains to show that $\Delta_A(M)$ is $\{U_{2,5}, U_{3,5}\}$-fragile, and that some element of $A$ is non-contractible in $\Delta_A(M)$. 
 
Since $A$ is an allowable segment of $M$, there is some $x\in A$ that is non-deletable in $M$. Then $M\del x$ is near-regular by Lemma \ref{nreg}, so $\Delta_{A-x}(M\del x)$ is also near-regular by Lemma \ref{3.7}. But $\Delta_A(M)/x=\Delta_{A-x}(M\del x)$ by Lemma \ref{2.13}, so $\Delta_A(M)/x$ has no $\{U_{2,5}, U_{3,5}\}$-minor by Corollary \ref{nreg}. Thus $x$ is non-contractible in $\Delta_A(M)$, so $A$ is an allowable cosegment of $\Delta_A(M)$. Now suppose that $\Delta_A(M)\del y$ has a $\{U_{2,5}, U_{3,5}\}$-minor for some $y\in A-x$. But $A-y$ is a non-trivial series class of $\Delta_A(M)\del y$, so $x$ is contractible in $\Delta_A(M)\del y$; a contradiction because $x$ is non-contractible in $\Delta_A(M)$. Thus each $y\in A-x$ is non-deletable in $\Delta_A(M)$.

If $\cl_{\si(M)}(A)$ contains at least 5 elements, then $M$ must have rank 2. Since $A$ is allowable, some element is non-deletable, so this closure has exactly five elements. Let $e \in \cl_M(A) - A$ be such that $M|(A\cup e) \cong U_{2,5}$. Since $A$ is coindependent, there is an $f \in M$ such that $\cl_M(\{e,f\}) = E(M)$. It is easy to see that, in $\Delta_M(A)$, element $e$ is not deletable, and $f$ is not contractible, so the result holds. Hence we may assume that $\cl_{\si(M)}(A)$ has at most four elements.

Suppose that $e\in E(M)-A$ is non-deletable in $M$. Then $M\del e$ is near-regular by Corollary \ref{nreg}. First, suppose that $A$ is not coindependent in $M\del e$. Let $B = E(M) - (A\cup e)$. Then $(A,B)$ is a 2-separation of $M\del e$, and either $|A| = 3$ or $|A| = 4$ and $B$ spans a point of $A$. Suppose $\Delta_A(M)\del e$ has a $\{U_{2,5},U_{3,5}\}$-minor $N$. Then $|E(N)\cap A| \leq 1$ or $|E(N)\cap B| \leq 1$. We claim that the former holds. Indeed: suppose $E(N)\cap B = \{f\}$. After deleting or contracting all elements of $B - \{f\}$ so that the $N$-minor is preserved, we are left with a matroid of rank 3 on 4 elements, or a matroid of rank 4 on 5 elements. Neither can have a minor isomorphic to $N$, a contradiction. Hence $|E(N)\cap A| \leq 1$. Suppose $x$ is deletable in $\Delta_A(M)\del e$. Then $A-x$ is a series class in this matroid, so we can contract all but one of the elements and still preserve $N$. But now we have obtained a matroid that is a minor of $M\del e$, a contradiction. Hence $x \in E(N)$. By the above, all $y \in A-x$ are contractible but not deletable. Once again, contracting them results in a minor of $M\del e$, a contradiction.

Now $A$ is a coindependent segment in $M\del e$. The matroid $\Delta_A(M\del e)$ is near-regular by Lemma \ref{3.7}. Thus $\Delta_A(M\del e)$ has no $\{U_{2,5},U_{3,5}\}$-minor. But $\Delta_A(M\del e)=\Delta_A(M)\del e$, so $e$ is non-deletable in $\Delta_A(M)$.

Next suppose that $e\in E(M)-A$ is deletable but non-contractible in $M$. First suppose that $e\in \cl_M(A)-A$. Seeking a contradiction, suppose that $e$ is contractible in $\Delta_A(M)$. Then $\Delta_A(M)/e$ has an $N$-minor for some $N\in \{U_{2,5}, U_{3,5}\}$. Now $(A,E(M)-A)$ is a $3$-separation of $\Delta_A(M)$ and $e$ is in the guts of $(A,E(M)-A)$, so $\Delta_A(M)/e$ has a $2$-separation $(A,B)$. Suppose $|B\cap E(N)| \leq 1$. Contract $e$ in $\Delta_A(M)$, and contract or delete all other members of $B$ preserving $N$, until only a single element $f \in B$ remains. Then we must have $|A| = 4$, and $E(N) = A\cup f$. But since $e$ is in parallel with some element of $A$ in $M$, there will be a triangle in the rank-3 matroid $N$, a contradiction. Hence $|A\cap E(N)| \leq 1$.

 Since $e$ is in the guts of the $3$-separation $(A,E(M)-A)$ in $\Delta_A(M)$, it follows that $A$ contains a circuit $C$ of $\Delta_A(M)/e$. Let $a\in A$ be an element in the circuit $C$. We claim that $a$ is flexible in $\Delta_A(M)$. We can assume that $\Delta_A(M)/e$ is the $2$-sum of matroids $M_A$ and $M_B$ with basepoint $p$. Let $C_a$ and $C_a^{*}$ be a maximum-sized circuit and cocircuit of $M_A$ containing $\{a,p\}$. If $|C_a|>2$, then $\Delta_A(M)/e/a$ has an $N$-minor, hence $a$ is contractible in $\Delta_A(M)$. Dually, if $|C_a^{*}|>2$, then $\Delta_A(M)/e\del a$ has an $N$-minor, hence $a$ is deletable in $\Delta_A(M)$. Thus $a$ is flexible unless $|C_a|=2$ or $|C_a^{*}|=2$. Suppose that $|C_a|=2$ or $|C_a^{*}|=2$. Then $a\in \cl^{(*)}_{\Delta_A(M)/e}(B)$. But $a\in \cl_{\Delta_A(M)/e}(A-a)$ since $a$ is in the circuit $C$, and $a\in \cl^{*}_{\Delta_A(M)/e}(A-a)$ because $A$ is a cosegment of $\Delta_A(M)/e$. Thus $a\notin \cl^{(*)}_{\Delta_A(M)/e}(B)$ by Lemma \ref{clandco}; a contradiction. Thus $e\in \cl_M(A)-A$ is non-contractible in $\Delta_A(M)$.
 
We may therefore assume that $e\notin \cl_M(A)-A$. Then $A$ is a coindependent segment of $M/e$ and $\Delta_A(M/e)$ is well-defined. Now $\Delta_A(M)/e=\Delta_A(M/e)$ by Lemma \ref{2.16}. But $M/e$ is near-regular, so $\Delta_A(M/e)$ is near-regular. Therefore $\Delta_A(M/e)$ has no $\{U_{2,5}, U_{3,5}\}$-minor by Corollary \ref{nreg}, and so $e$ is non-contractible in $\Delta_A(M)$. 
\end{proof}

\subsection{Gluing on wheels}
Let $M$ be a matroid, and $(a,b,c)$ an ordered subset of $E(M)$ such that $T = \{a,b,c\}$ is a triangle. Let $N$ be obtained from $M(\mathcal{W}_r)$ by relabeling some triangle as $\{a,b,c\}$, where $a,c$ are spoke elements, and let $X \subseteq \{a,b,c\}$ such that $b \in X$. Following the terminology from \cite{chun2013fan}, we say the matroid $M' := P_T(M,N) \delete X$ was obtained from $M$ by \emph{gluing an $r$-wheel onto $(a,b,c)$}. We also say that $M^{*}$ is obtained from $N^{*}$ by gluing a wheel to the triad $T$. 

\begin{lemma}
\label{fanfragile}
Let $\mathbb{P}\in \{\mathbb{U}_2, \mathbb{H}_5\}$, and let $M$ be a strictly $\{U_{2,5}, U_{3,5}\}$-fragile $\mathbb{P}$-representable matroid. Let $A = \{a,b,c\}$ be an allowable triangle of $M$, such that $b$ is non-deletable. 
Let $M'$ be obtained from $M$ by gluing an $r$-wheel onto $(a,b,c)$, where $X \subseteq \{a,b,c\}$ is as above. If $M'$ is $3$-connected, then $M'$ is a strictly $\{U_{2,5}, U_{3,5}\}$-fragile $\mathbb{P}$-representable matroid. Moreover, $F = E(\mathcal{W}_r) - X$ is the set of elements of a fan, the spoke elements of $F$ are non-contractible in $M'$, and the rim elements of $F$ are non-deletable in $M'$. 
\end{lemma}


The proof hinges on the following observations, whose proofs we omit.

\begin{lemma}\label{lem:chaingpc}
  Let $M_1, M_2, M_3$ be matroids, such that $E(M_1)\cap E(M_2) = A$, $E(M_2)\cap E(M_3) = B$, $E(M_1)\cap E(M_3) = \emptyset$, and such that $A$ is a modular flat of $M_2$, and $B$ a modular flat of $M_3$. Then
  \begin{align*}
    P_A(M_1, P_B(M_2,M_3)) = P_B(M_3, P_A(M_1,M_2)).
  \end{align*}
\end{lemma}

\begin{lemma}\label{lem:growwheel}
  Let $N'$ be isomorphic to the $r$-wheel, with $T$ a triangle. Let $N$ be obtained from $N'$ by adding elements in parallel to each of the spokes in $T$. Let $M$ be isomorphic to $M(\mathcal{W}_3)$ with a triangle labeled by $T$. Then $P_T(N, M)\delete T = \Delta_{T}(N)$ is isomorphic to $M(\mathcal{W}_{r+1})$.
\end{lemma}

\begin{proof}[Proof of Lemma \ref{fanfragile}]
We proceed by induction on the rank $r$ of the wheel glued onto $A=(a,b,c)$. Since representability and fragility are minor-closed properties, it suffices to prove the result when the set of deleted elements when gluing is $X = \{b\}$. Let $M_r$ be the matroid obtained by gluing an $r$-wheel onto $M$ and deleting $X$. For the base case, observe that since $A$ is allowable, $a$ and $c$ are non-contractible. Let $M_+$ be the matroid obtained from $M$ by adding elements $a', c'$ in parallel with $a$ and $c$ respectively, and consider $M_3 = P_{\{a',b,c'\}}(M_+, W) \del \{a',b,c'\} = \Delta_{\{a',b,c'\}}(M_+)$. The conclusion now follows from Lemma \ref{amove}.

Assume the result holds for $M_r$. Let $T$ be any triangle of $M(\mathcal{W}_r)$ other than $\{a,b,c\}$. Clearly $T$ is allowable. Now apply the base case with $M$ replaced by $M_r$. This matroid, $M_{r+1}'$ say, satisfies all conclusions of the lemma. That $M_{r+1}'$ is isomorphic to $M_{r+1}$ follows from Lemmas \ref{lem:chaingpc} and \ref{lem:growwheel}.
\end{proof}

\section{Small cases and fan-extensions}\label{sec:noX}
In \cite{chunfragile}, we used a computer to enumerate all 3-connected, strictly $\{U_{2,5},U_{3,5}\}$-fragile, $\mathbb{H}_5$-representable matroids on up to nine elements. Figures \ref{fig:H5_5}--\ref{fig:H5_9} give some geometric representations, and Figure \ref{fig:M71} gives a labeling of the elements of some of them. We refer the reader to \cite{chunfragile} for a more detailed description of the construction of the matroids in Figures \ref{fig:H5_5}--\ref{fig:M71} from either $U_{2,5}$ or $M_{7,1}$.

In \cite{chunfragile} we apply the main result of \cite{chun2013fan} together with a case analysis (carried out by a computer) to prove the following Theorem.
\begin{thm}
 \cite[Theorem 1.3]{chunfragile}
 \label{computer1}
  Let $M'$ be a $3$-connected strictly $\{U_{2,5}, U_{3,5}\}$-fragile $\mathbb{H}_5$-representable matroid. Then $M'$ is isomorphic to a matroid $M$ for which one of the following holds:
  \begin{enumerate}
  \item[(i)] $M$ has an $\{X_8,Y_8,Y_8^{*}\}$-minor;
  \item[(ii)] $M\in \{U_{2,6},U_{4,6},P_6, M_{9,9}, M_{9,9}^{*}\}$;
  \item[(iii)] $M$ or $M^{*}$ can be obtained from $U_{2,5}$ (with groundset $\{a,b,c,d,e\}$) by gluing wheels to $(a,c,b)$,$(a,d,b)$,$(a,e,b)$;
  \item[(iv)] $M$ or $M^{*}$ can be obtained from $U_{2,5}$ (with groundset $\{a,b,c,d,e\}$) by gluing wheels to $(a,b,c)$,$(c,d,e)$;
  \item[(v)] $M$ or $M^{*}$ can be obtained from $M_{7,1}$ by gluing a wheel to $(1,3,2)$.
  \end{enumerate}
\end{thm}
We also obtain the following Corollary of Theorem \ref{computer1}.
\begin{cor}
 \cite[Corollary 1.4]{chunfragile}
 \label{computer2}
  Let $M'$ be a $3$-connected strictly $\{U_{2,5}, U_{3,5}\}$-fragile $\mathbb{U}_2$-representable matroid. Then $M'$ is isomorphic to a matroid $M$ for which one of the following holds: 
    \begin{enumerate}
  \item[(i)] $M$ has an $\{X_8,Y_8,Y_8^{*}\}$-minor;
  \item[(ii)] $M\in \{M_{9,9}, M_{9,9}^{*}\}$;
  \item[(iii)] $M$ or $M^{*}$ can be obtained from $U_{2,5}$ (with groundset $\{a,b,c,d,e\}$) by gluing wheels to $(a,c,b)$,$(a,d,b)$,$(a,e,b)$;
  \item[(iv)] $M$ or $M^{*}$ can be obtained from $U_{2,5}$ (with groundset $\{a,b,c,d,e\}$) by gluing wheels to $(a,b,c)$,$(c,d,e)$;
  \item[(v)] $M$ or $M^{*}$ can be obtained from $M_{7,1}$ by gluing a wheel to $(1,3,2)$.
  \end{enumerate}
\end{cor}


\begin{figure}[htbp]
  \centering
    \includegraphics[scale=1]{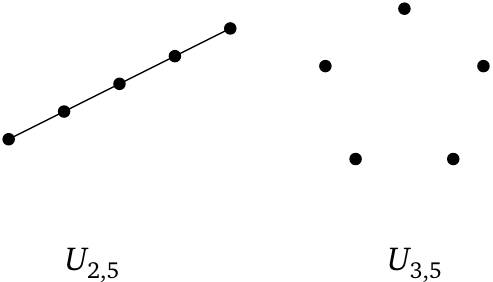}
  \caption{The 3-connected $\mathbb{H}_5$-representable fragile matroids on 5 elements.}
  \label{fig:H5_5}
\end{figure}

\begin{figure}[htbp]
  \centering
    \includegraphics[scale=1]{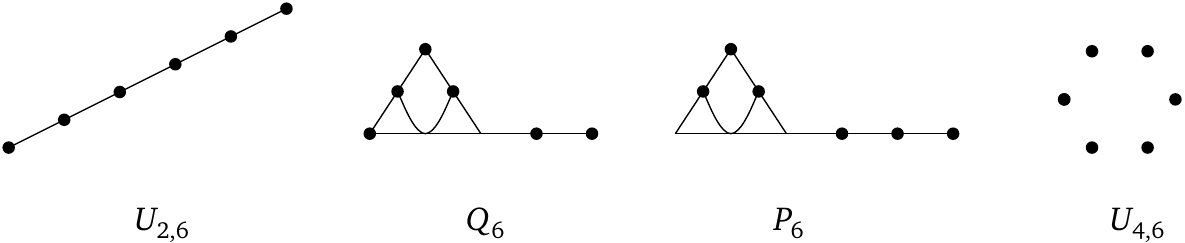}
  \caption{The 3-connected $\mathbb{H}_5$-representable fragile matroids on 6 elements.}
  \label{fig:H5_6}
\end{figure}

\begin{figure}[htbp]
  \centering
    \includegraphics[scale=1]{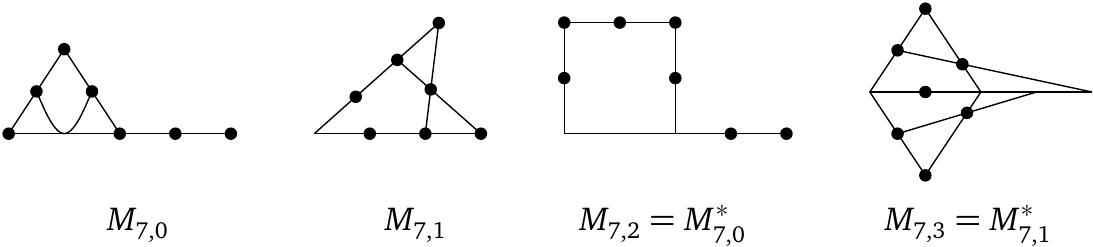}
  \caption{The 3-connected $\mathbb{H}_5$-representable fragile matroids on 7 elements.}
  \label{fig:H5_7}
\end{figure}

\begin{figure}[htbp]
  \centering
    \includegraphics[scale=1]{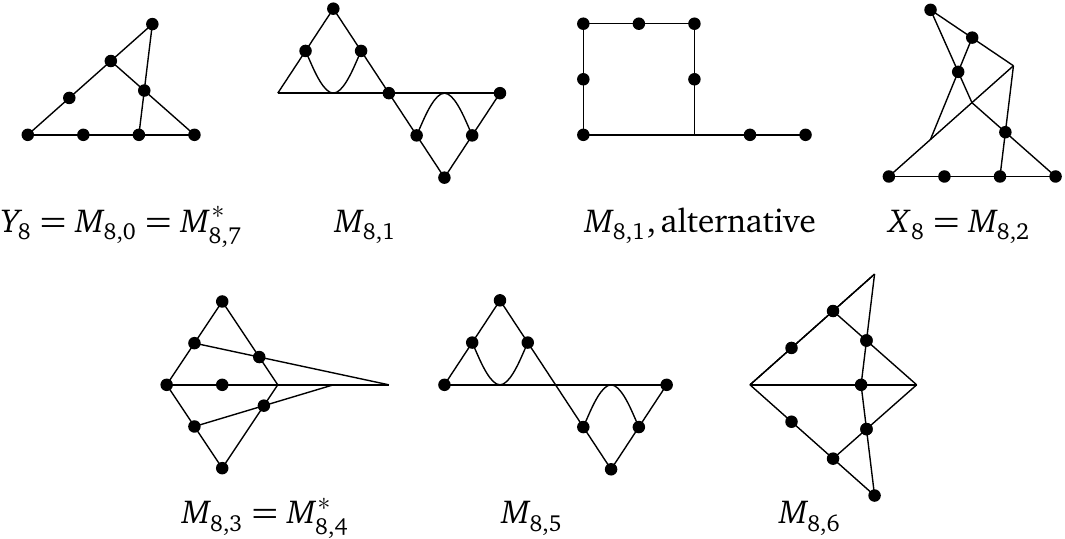}
  \caption{The 3-connected $\mathbb{H}_5$-representable fragile matroids on 8 elements.}
  \label{fig:H5_8}
\end{figure}

\begin{figure}[htbp]
  \centering
    \includegraphics[scale=1]{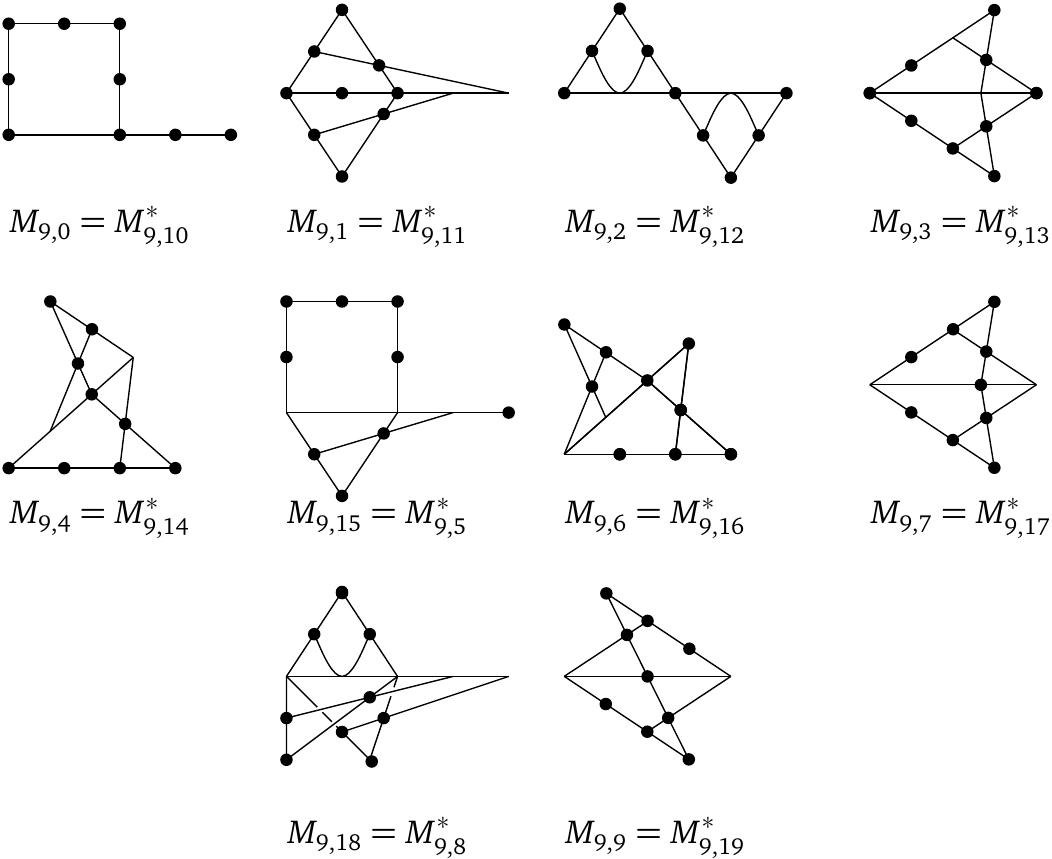}
  \caption{The 3-connected $\mathbb{H}_5$-representable fragile matroids on 9 elements.}
  \label{fig:H5_9}
\end{figure}

\begin{figure}[htbp]
  \centering
    \includegraphics[scale=1]{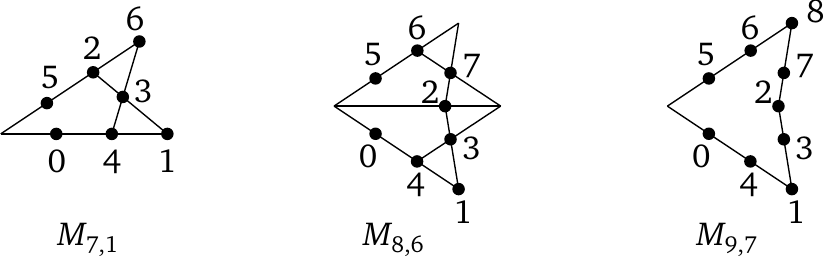}
  \caption{The matroids $M_{7,1}$, $M_{8,6}$, and $M_{9,7}$. In the right-most diagram, the 2-point lines were omitted to emphasize the fan $(1,3,2,7,8)$.} 
  \label{fig:M71}
\end{figure}

\section{The main result}
\label{sec:X}

In this paper we augment Threorem \ref{computer1} and Corollary \ref{computer2} by giving a constructive description of the $3$-connected $\mathbb{U}_2$- and $\mathbb{H}_5$-representable strictly $\{U_{2,5}, U_{3,5}\}$-fragile matroids that have an $\{X_8,Y_8,Y_8^{*}\}$-minor. Roughly speaking, in addition to gluing on wheels as in the other classes, we can also build paths of $3$-separations along the $4$-element segments and cosegments of the matroids in $\{X_8,Y_8,Y_8^{*}\}$. In order to describe the structure of these matroids formally we need some more definitions.

A set $X \subseteq E(M)$ is \textit{fully closed} if $X$ is closed in both $M^{*}$ and $M$. The \textit{full closure} of $X$, denoted $\fcl_M(X)$ is the intersection of all fully closed sets containing $X$. It can be obtained by repeatedly taking closures and coclosures until no new elements are added. We call $X$ a \textit{path-generating} set if $X$ is a $3$-separating set of $M$ such that $\fcl_M(X)=E(M)$. For example, the $4$-element segments of $Y_8$ are path-generating sets, but the triangles of $Y_8$ that meet both the $4$-element segments are not path-generating sets. It is easy to see that the notion of path-generating sets is invariant under duality.

A path-generating set $X$ thus gives rise to a natural path of $3$-separating sets $(P_1,\ldots, P_m)$, where $P_1 = X$ and each step $P_i$ is either the closure or coclosure of the $3$-separating set $P_{i-1}$. The path-generating sets in this paper will be allowable sets or fans.

Let $X$ be an $\mathcal{N}$-allowable cosegment $M$. A matroid $Q$ is an \textit{$\mathcal{N}$-allowable series extension along $X$} if $M=Q/Z$ and, for every element $z$ of $Z$ there is some element $x$ of $X$ such that $x$ is $\mathcal{N}$-contractible in $M$ and $z$ is in series with $x$ in $Q$. We also say that $Q^{*}$ is an \textit{$\mathcal{N}^{*}$-allowable parallel extension along $X$}.

In what follows, $S$ will be the elements of the 4-element segment of $X_8$, and $C$ the elements of the 4-element cosegment of $X_8$, so $E(X_8) = S \cup C$. We say that $M$ is obtained from $N$ by a \textit{\dy-step} along $A$ if $A \in \{S,C\}$ and, up to duality, $M$ is obtained from $N$ by performing a non-empty allowable parallel extension along a path-generating allowable set $A$ of $N$, followed by a generalized $\Delta-Y$ exchange on $A$. Note that, in $N$, each of $S$ and $C$ can be either a segment or a cosegment.

We also use the notation $M=\Delta_{A}^Q(N)$ and $N=\nabla_{A}^Q(N)$ if there is some allowable parallel or series extension $Q$ of $N$ along $A$ such that $M=\Delta_{A}(Q)$ or $M=\nabla_{A}(Q)$.

A sequence of matroids $M_1,\ldots, M_n$ is called a \textit{path sequence} if the following conditions hold:
\begin{enumerate}
 \item[(P1)] $M_1=X_8$; and
 \item[(P2)] For each $i\in \{1,\ldots, n-1\}$, there is some $4$-element path-generating segment or cosegment $A \in \{S,C\}$ of $M_i$ such that either: 
\begin{itemize}
 \item[(a)] $M_{i+1}$ is obtained from $M_{i}$ by a \dy-step along $A$; or
 \item[(b)] $M_{i+1}$ is obtained from $M_i$ by gluing a wheel onto an allowable subset $A'$ of $A$.
\end{itemize}
\end{enumerate}

Note in (P2) that each \dy-step described in (a) increases the number of elements by at least one, and that the wheels in (b) are only glued onto allowable subsets of 4-element segments or cosegments.

We say that a path sequence $M_1,\ldots, M_n$ \textit{describes} a matroid $M$ if $M_n\cong M$. We also say that $M$ is a matroid \textit{described by} a path sequence if there is some path sequence that describes $M$. Let $\mathcal{P}$ denote the class of matroids such that $M\in \mathcal{P}$ if and only if there is some path sequence $M_1,\ldots, M_n$ that describes a matroid $M'$ such that $M$ can be obtained from $M'$ by $0$ or more allowable parallel and series extensions. Since $X_8$ is self-dual it is easy to see that the sequence of dual matroids $M_1^{*},\ldots, M_n^{*}$ of a path sequence $M_1,\ldots, M_n$ is also a path sequence.  Thus the class $\mathcal{P}$ is closed under duality. 


We say that $M_1,\ldots, M_n$ is a \textit{path sequence ending in a $4$-element segment $A$} if $M_{n}$ is obtained from $M_{n-1}$ by a \dy-step along $A$, and $A$ is a $4$-element segment of $M_n$. We say that $M_1,\ldots, M_n$ is a \textit{path sequence ending in a $4$-element cosegment $A$} if $M_1^{*},\ldots, M_n^{*}$ is a path sequence ending in a $4$-element segment $A$. 

Suppose that  $M_1,\ldots, M_n$ is a path sequence such that $M_n$ is obtained from $M_{n-1}$ by gluing a wheel $W$ onto an allowable triangle $A'$ of $M_{n-1}$. Then $M_n=P_{A'}(M_{n-1},W)\del X$ for some subset $X$ of $A'$ containing the unique rim element of $A'$. Let $F=E(W)-X$. Then we call $M_1,\ldots, M_n$ (and $M_1^{*},\ldots, M_n^{*}$) a \textit{path sequence ending in a fan $F$}. We also call $M_1,\ldots, M_n$ a \textit{path sequence ending in a triangle $F$} or a \textit{path sequence ending in a triad $F$} if $M_1,\ldots, M_n$ is a path sequence ending in a fan $F$, where the fan $F$ is a triangle or triad of $M_{n}$ respectively.

Given a path sequence $M_1,\ldots, M_n$, a path of $3$-separations can be determined as follows: for $X_8$, we let $(S,C)$ be the path of $3$-separations. Otherwise $n\geq 2$, so we can assume by duality that either $M_1,\ldots, M_n$ is a path sequence ending in a $4$-element segment $A$ or else $M_1,\ldots, M_n$ is a path sequence ending in a fan $A$, where $M_n$ is obtained from $M_{n-1}$ by gluing a wheel onto an allowable triangle of $M_{n-1}$. We set $P_1=A$, and then define the remaining steps of the path of $3$-separations by alternately taking the closure and coclosure respectively until no new steps can be added. That is, $P_k=\cl(P_1\cup \cdots P_{k-1})$ for $k\geq 2$ and $k$ even, and $P_k=\cl^{*}(P_1\cup \cdots P_{k-1})$ for $k\geq 2$ and $k$ odd. We call the resulting path of $3$-separations $\textbf{P}=(P_1,\ldots, P_m)$ the \textit{path of $3$-separations associated with the path sequence} $M_1,\ldots, M_n$. A $3$-separation $(R,G)$ of $M_n$ is \textit{displayed by} $\textbf{P}$ if $(R,G)=(P_1\cup \cdots \cup P_{i-1},P_{i}\cup \cdots \cup P_n)$ for some $i\in \{2,\ldots, n\}$.

The following is an easy consequence of the definition of a path sequence, together with Lemma \ref{amove} and Lemma \ref{fanfragile}.


\begin{lemma}
\label{pathsequencefragile}
 Let $\mathbb{P}\in \{\mathbb{U}_2, \mathbb{H}_5\}$. If $M$ is a $3$-connected matroid described by a path sequence, then $M$ is a strictly $\{U_{2,5}, U_{3,5}\}$-fragile $\mathbb{P}$-representable matroid. Moreover, $M$ has a path of $3$-separations $\mathbf{P}=(P_1,\ldots, P_m)$ such that $P_1$ and $P_m$ are path-generating allowable sets, and, for each $i\in \{2,\ldots, m-1\}$, the elements of $P_i$ are in the guts or coguts of the $3$-separation $(P_1\cup \cdots \cup P_{i-1},P_{i}\cup \cdots \cup P_n)$ and $|P_i|\leq 3$.
\end{lemma}


We can now state the main result of the paper.

\begin{thm}\label{thm:mainresult}
If $M$ is a $3$-connected $\{U_{2,5}, U_{3,5}\}$-fragile $\mathbb{H}_5$-representable matroid that has an $\{X_8,Y_8,Y_8^{*}\}$-minor, then there is some path sequence that describes $M$.
\end{thm}

The proof will take up the remaining sections. We also have the following corollary for the class of $\mathbb{U}_2$-representable matroids.

\begin{cor}
If $M$ is a $3$-connected $\{U_{2,5}, U_{3,5}\}$-fragile $\mathbb{U}_2$-representable matroid that has an $\{X_8,Y_8,Y_8^{*}\}$-minor, then there is some path sequence that describes $M$.
\end{cor}

\section{Path sequence rearrangement}

Path sequence descriptions are generally not unique. That is, we can have distinct path sequences $M_1,\ldots, M_n$ and $M_1',\ldots, M_n'$ such that $M_n\cong M_n'$. One reason for this is that $X_8$ has two disjoint path-generating sets, and the order in which these are used does not change the resulting matroid described by a path sequence. Another reason is because of symmetry. We will focus on the ends of path sequences in the following sections, and we often speak of path sequences with a particular ordering of steps. In this section we show when we can obtain a path sequence with a desired ordering. The results found here are mostly routine and follow from corresponding properties of the generalized $\Delta-Y$ exchange. 

We introduce some terminology here to avoid cumbersome statements, but note that these terms will not appear in the following sections. We call a path sequence a \textit{\dy-sequence} if only moves of type (P2)(a) are used. That is, if no step involves gluing a wheel onto an allowable set. Let $M_1,\ldots, M_n$ be a \dy-sequence, and remember that $S$ and $C$ are the $4$-element segment and cosegment of $X_8$ respectively. There is an associated sequence $a[M_1,\ldots, M_n]: \{2,\ldots, n\}\to \{S,C\}$ for $M_1,\ldots, M_n$ where, for $i\in \{2,\ldots n\}$, we define $a[M_1,\ldots, M_n](i)=X$ if $M_{i}$ is obtained from $M_{i-1}$ by using the set $X$. We call $a[M_1,\ldots, M_n]$ the \textit{adjacency sequence} for $M_1,\ldots, M_n$. The adjacency sequence records the order of the steps used by the path sequence. If $n = 2$ then either $S$ and $C$ are both segments of $M_2$ or both cosegments of $M_2$. In the former case, $a[M_1,M_2](2) = (C)$, and in the latter, $a[M_1,M_2](2) = S$. When there are two or more steps, there is a lot of flexibility in the adjacency sequence. First, we show that moves on disjoint path-generating sets commute.

\begin{lemma}
\label{commute1}
Let $M_1,\ldots, M_n$ be a \dy-sequence. If $a[M_1,\ldots, M_n](i) \neq a[M_1,\ldots, M_n](i+1)$ for some $i\in \{2,\ldots,n\}$, then there is some \dy-sequence $M_1',\ldots, M_n'$ such that $M_n=M_n'$, and $a[M_1',\ldots, M_n'](j)\neq a[M_1,\ldots, M_n](j)$ if and only if $j\in \{i,i+1\}$.
\end{lemma}

\begin{proof}
Suppose that $a[M_1,\ldots, M_n](i)=C$ and $a[M_1,\ldots, M_n](i+1)=S$ for some $i\in \{2,\ldots,n\}$. Suppose that $M_i=\Delta_C^{N_i}(M_{i-1})$ and $M_{i+1}=\Delta_S^{N_{i+1}}(M_i)$. Let $Q=E(N_i)-E(M_{i-1})$ and $R=E(N_{i+1})-E(M_i)$, and consider the matroid $P=\nabla_C(\nabla_S(M_{i+1}))$. Then $P$ is a parallel extension of $M_{i-1}$ with the elements of $Q$ along $C$ and the elements of $R$ along $S$. Now $P_{i}=P\del Q$ is an allowable parallel extension of $M_{i-1}$ along $S$, so $M_{i}'=\Delta_S^{P_i}(M_{i-1})$ is obtained from $M_{i-1}$ using $S$. Next $P_{i+1}=\Delta_S(P)$ is an allowable parallel extension of $M_{i}'$ along $C$, so $M_{i+1}'=\Delta_C^{P_{i+1}}(M_{i}')$ is obtained from $M_i'$ using $C$. But now $M_{i+1}'=\Delta_C(\Delta_S(P))$, so $M_{i+1}'= M_{i+1}$ by Lemma \ref{2.18}. The argument when $M_i$ or $M_{i+1}$ use the $\nabla$ operation is symmetric, with Corollary \ref{2.19} used in place of Lemma \ref{2.18}. 
\end{proof}

The following lemma is a key to exchanging the order of the \dy-steps of a \dy-sequence. 

\begin{lemma}\label{lem:SCexchange}
  Let $M$ be such that $M = \Delta^Q_S(X_8)$ or $M = \nabla^Q_C(X_8)$. Then $M$ has an automorphism exchanging $S$ and $C$.
\end{lemma}

\begin{proof}
  Assume $M = \Delta^Q_S(X_8)$. The other case follows by duality.
  We have $M = \Delta^Q_S(X_8) = \Delta^Q_S(\Delta^R_C(U_{2,5}))$, where $R$ is an allowable parallel extension of $U_{2,5}$ by three elements. For concreteness, let $S = \{s_1, s_2,s_3,s_4\}$, let $C = \{c_1,c_2,c_3,c_4\}$, suppose $E(U_{2,5}) = \{s_1, s_2, s_3, s_4, c_4\}$, and $E(R) - E(S) = \{c_1,c_2,c_3\}$ with $\{s_i,c_i\}$ a parallel pair for $i = 1,2,3$. Now $Q$ is an allowable parallel extension of $X_8$, and up to symmetry we may assume it was obtained by placing elements $p_i$ in parallel with $s_i$, for $i \subseteq \{1,2,3\}$. Let $Q'$ be the matroid obtained from $U_{2,5}$ by doing both parallel extensions, i.e. $Q' = \nabla_C(Q)$. Clearly $Q'$ has an automorphism exchanging $S$ and $C$, and therefore $M = \Delta_S(\Delta_C(Q')) = \Delta_C(\Delta_S(Q'))$ also has the desired automorphism.
\end{proof}

Next, we show that we can exchange an $S$ and a $C$. We omit the dual statement of the following result. 

\begin{corollary}
 \label{usingeachend2}
Let $M_1,\ldots, M_n$ be a \dy-sequence such that $a[M_1, \ldots, M_n]$ has $k$ entries equal to $C$ and $n-k-1$ entries equal to $S$. If $k \geq 1$, then there is some equivalent $\dy$-sequence $M_1', \ldots, M_n'$ such that 
$a[M_1', \ldots, M_n']$ has $k-1$ entries equal to $S$ and $n-k$ entries equal to $C$.
\end{corollary}

\begin{proof}
  This follows from Lemma \ref{lem:SCexchange} when $n \leq 2$. For $n \geq 3$, apply Lemma \ref{commute1} to assume $a[M_1,\ldots, M_n](2,3) = (C,C)$ or  $a[M_1,\ldots, M_n](2,3) = (S,C)$. By Lemma \ref{lem:SCexchange}, there is an automorphism of $M_2$ that exchanges $S$ and $C$. Let $M_2'$ be the image of this automorphism. Applying it to all other $M_j$, we find a $\dy$-sequence $M_1', M_2', \ldots, M_n'$ whose adjacency sequence starts with $(C,S)$ or $(S,S)$ respectively. Note that all remaining entries get swapped from $C$ to $S$ and vice versa. This yields the desired result.
\end{proof}

We now return to considering general path sequences. Since, by Lemma \ref{lem:growwheel}, gluing a wheel onto an allowable triangle contained in an allowable $4$-element segment is the same as performing a sequence of moves, where each move consists of allowable parallel-extensions followed by a $\Delta$-$Y$ exchange, it again follows easily from Lemma \ref{2.18} and Corollary \ref{2.19} that this operation commutes with moves described in (P2) that are performed on a disjoint path-generating set. When a wheel is glued onto on an allowable subset of a $4$-element segment or cosegment $A$, it is clear that $A$ can no longer be a $4$-element segment or cosegment, so no additional \dy-step can be performed on $A$.
Thus we have the following lemma, whose proof is similar to the last. 

\begin{lemma}
\label{commute3}
Let $M_1,\ldots, M_n$ be a path sequence that describes $M$, and let $F$ be a fan of $M$. If there is some $k\in \{1,\ldots, n\}$ such that $M_1,\ldots, M_k$ is a path sequence ending in the fan $F$, then $M$ can be described by a path sequence $M_1',\ldots, M_n'$ ending in the fan $F$. Moreover, if $M_1',\ldots, M_{n-1}'$ is a \dy-sequence with at least one \dy-step, then $M$ can be described by a path sequence $M_1'',\ldots, M_n''$ ending in either a $4$-element segment or cosegment $A$ of $M$, where $A\in \{S,C\}$ and $A$ is disjoint from $F$.
\end{lemma}

Let $M$ be a matroid described by a path sequence $M_1,\ldots, M_n$. If $M_1,\ldots, M_n$ has a subsequence of the form $M_{i}=\Delta_A^Q(M_{i-1})$ and $M_{i+1}=\nabla_A^R(M_{i})$, then $r\in E(R)-E(M_{i})$ is called an \textit{internal coguts element} and $q\in E(Q)-E(M_{i-1})$ is called an \textit{internal guts element}. Dually, for the subsequence $M_{i}^{*}$ and $M_{i+1}^{*}$ of $M_1^{*},\ldots, M_n^{*}$, we call $r\in E(R)-E(M_{i}^{*})$ an \textit{internal guts element} and $q\in E(Q)-E(M_{i-1}^{*})$ an \textit{internal coguts element}. Internal guts and coguts elements can be removed without disturbing the end steps of the path sequence, in the following sense. 

\begin{lemma}
\label{Pmintcoguts}
Let $M$ be a matroid described by a path sequence $M_1,\ldots, M_n$ ending in either a $4$-element segment, $4$-element cosegment, or fan $X$. If $q\in E(M)$ is an internal guts element and $r\in E(M)$ is an internal coguts element, then $M\del q\in \mathcal{P}$ and $M/r\in \mathcal{P}$. Moreover, there are matroids $M'$ and $M''$ that can each be described by path sequences ending in the $4$-element segment, $4$-element cosegment, or fan $X$ such that $M\del q$ is a series extension of $M'$ and $M/r$ is a parallel extension of $M''$. 
\end{lemma}

\begin{proof}
Suppose some subsequence of steps of the path sequence has the form $M_{i}=\Delta_A^Q(M_{i-1})$ and $M_{i+1}=\nabla_A^R(M_{i})$. Let $r\in E(R)-E(M_{i})$. Then either $R/r=M_{i}$ or $R/r$ is an allowable series extension of $M_{i}$ along $A$. Suppose first that $R/r=M_{i}$. Then $M_{i+1}/r=Q$ by Lemma \ref{2.11}, so $M_{i+1}/r$ is a parallel extension of a matroid described by a path sequence ending in the $4$-element segment $A$. Now consider contracting the element $r$ of $M$. The contraction operation commutes with any subsequent steps of the path sequence that describes $M$ by Lemma \ref{2.16}, so $M/r$ is also is a parallel extension of a matroid described by a path sequence ending in $X$. For the second case, suppose that $R/r$ is an allowable series extension of $M_{i}$. Consider contracting $r$ from $M$. Again the contraction of $r$ commutes with any subsequent steps of the path sequence that describes $M$ by Lemma \ref{2.16}, and the $(i+1)$-th step becomes $\nabla_A^{R/r}(M_{i})$. It follows that $M/r$ is described by a path sequence ending in $X$. 

Similarly, for $q\in E(Q)-E(M_{i-1})$, either $Q\del q=M_{i-1}$ or $Q\del q$ is an allowable parallel extension of $M_{i-1}$. Suppose that $Q\del q=M_{i-1}$. Then $M_i\del q=\Delta_A(M_{i-1})$ by Lemma \ref{2.11}. Thus $M_{i+1}\del q$ is equal to $M_{i-1}$ up to series pairs by Lemma \ref{2.11}, and since the deletion of $q$ commutes with any subsequent steps of the path sequence that describes $M$ by Lemma \ref{2.16}, it follows that $M\del q\in \mathcal{P}$ and $M\del q$ is a series extension of a matroid described by a path sequence ending in $X$. Suppose that $Q\del q$ is an allowable parallel extension of $M_{i-1}$. Then the deletion of $q$ commutes with any subsequent steps of the path sequence that describes $M$ by Lemma \ref{2.16}, and the $(i+1)$-th step becomes $\Delta_A^{Q\del q}(M_{i-1})$, so $M\del q$ is described by a path sequence ending in $X$.
\end{proof}

Next we handle elements belonging to path-generating sets for matroids described by long path sequences.

\begin{lemma}
\label{parallelendelements}
Let $M$ be a matroid described by a path sequence $M_1,\ldots, M_n$ ending in a $4$-element segment $A$, where $M_1,\ldots, M_n$ has at least two \dy-steps. If $s\in E(M)-A$ is in parallel with an element of $A$, then $M\del s$ is described by a path sequence $M_1',\ldots, M_n'$ ending in $A$.
\end{lemma}

\begin{proof}
Let $s\in E(M)-A$ be an element in parallel with an element of $A$. If $s$ were added in some \dy-step, then, since $M_1,\ldots, M_n$ has at least two \dy-steps, it follows from Corollary \ref{usingeachend2} that $M$ is described by a path sequence $M_1',\ldots, M_n'$ where $s$ is an internal guts element of $M_1',\ldots, M_n'$ and the result would follow immediately from Lemma \ref{Pmintcoguts}. We may therefore assume that $s$ belongs to a maximal allowable path-generating set $X$ contained in $E(M_n)-A$. Now it follows from the path sequence construction that $s$ is the only element of $X$ in parallel with $A$. Moreover, it follows by orthogonality that $X$ is either a $4$-element segment or a fan with $|X|\geq 4$ such that $s$ is an end spoke element of $X$. We handle the two cases now.

Suppose that $X$ is a $4$-element segment. Since $M_1,\ldots, M_n$ has at least two \dy-steps, it follows from Corollary \ref{usingeachend2} that $M$ is described by a path sequence $M_1',\ldots, M_n'$ ending in $X$, so that $M_n'=\nabla_X^N(M_{n-1}')$. Hence $M\del s=\nabla_{X-s}(N/s)$ by Lemma \ref{2.13}, where $(X-s)\cup (E(N)-(E(M_{n-1}'\cup s))$ is a fan of $M\del s$. Thus $M\del s$ has a path sequence $M_1',\ldots, M_{n-1}',M_n''$, where $M_n''$ is obtained from $M_{n-1}'$ by gluing a wheel onto the triangle $X-s\subseteq X$ of $M_{n-1}'$. Since $M_1',\ldots, M_{n-1}',M_n''$ must have at least one \dy-move, it follows by Lemma \ref{commute3} that $M\del s$ has a path sequence ending in $A$.

Suppose that $X$ is a fan with $|X|\geq 4$ such that $s$ as an end spoke element. By Lemma \ref{commute3} it follows that $M$ is described by a path sequence $M_1',\ldots, M_n'$ such that $M_n'$ is obtained from $M_{n-1}'$ by gluing a wheel onto an allowable triangle of $S$, where $X$ is the corresponding fan. Since $s$ is an end spoke element of $X$ and $|X|\geq 4$, it follows that $X-s$ is a fan of $M\del s$. Moreover, it is easy to see that $M\del s$ has a path sequence $M_1',\ldots, M_{n-1}', M_n''$ where $M_n''$ is obtained from $M_{n-1}'$ by gluing a wheel onto an allowable triangle of $S$, where $X-s$ is the corresponding fan. Since $M_1',\ldots, M_{n-1}', M_n''$ has at least one \dy-move, it follows from Lemma \ref{commute3} that $M\del s$ has a path sequence ending in $A$.
\end{proof}

Let $M\in \mathcal{P}$, and let $A$ be a path-generating allowable $4$-element segment of $M$. We say that $A$ is \textit{parallel} to a $4$-element segment $A'$ if $A'$ is obtained by possibly replacing one or more deletable elements $a$ of $A$ by an element $a'$ in parallel with $a$ in $M$. It is easy to see that the parallel relation is an equivalence relation on the path-generating allowable $4$-element segments of $M$, and that if $A$ and $A'$ are equivalent, then $\Delta_A(M)\cong  \Delta_{A'}(M)$. 

\begin{lemma}
\label{tech4seg}
Let $M$ be a matroid described by a path sequence $M_1,\ldots, M_n$ that has at least one \dy-step, and let $A$ be a path-generating allowable $4$-element segment of $M$. There is a path sequence that describes $M$ such that $A$ is parallel to a $4$-element segment in $\{S,C\}$. 
\end{lemma}

\begin{proof}
Assume that neither $A$ nor any $4$-element segment parallel to $A$ belongs to $\{S,C\}$. Then since $M$ is a matroid described by a path sequence, it follows that $A=T\cup \{t\}$ where $T$ is a path-generating allowable triangle contained in either $S$ or $C$, and $t$ is a deletable element of $M$. We may assume that $T\subseteq C$ by Corollary \ref{usingeachend2}. 

Now, by Lemma \ref{commute3}, we may assume that $C$ is a $4$-element cosegment of $M_{n-1}$ and that $M=\nabla_T(M_{n-1})$. We show that we can construct a new path sequence that describes $M$ ending in a $4$-element segment parallel to $A$. Let $N$ be the series extension of $M_{n-1}$ obtained by adding an element in series with $c\in C-T$. Let $M_n'=\nabla_C^N(M_{n-1})$. By the dual of Lemma \ref{2.13} it follows that $\nabla_C(N)\del c= \nabla_{C-c}(N/c)=M$, so $M_n'$ is a parallel extension of $M$ where $c$ is in parallel with the element $t$. Thus $C$ is parallel to $A$ in $M_n'$. Moreover, since $M_1,\ldots, M_n'$ is a path sequence with at least two \dy-steps, it follows from Lemma \ref{parallelendelements} that $M_n'\del t=M$ is described by a path sequence ending in the $4$-element segment $C$, as required.
\end{proof}

We have the following important consequence of Lemma \ref{tech4seg}. 

\begin{lemma}
\label{4segrecognition}
Let $M\in \mathcal{P}$, and let $A$ be a path-generating allowable $4$-element segment of $M$. If $N$ is an allowable parallel extension of $M$ along $A$, then $\Delta_A(N)\in \mathcal{P}$.
\end{lemma}

\begin{proof}
Since $M\in \mathcal{P}$, we know that $M$ is some series-parallel extension of a matroid $M'$ described by a path sequence, so up to removing parallel or series elements and replacing $A$ by an equivalent $4$-element segment if necessary, we can assume that there is some path sequence $M_1,\ldots, M_n$ that describes $M$. Now, if $M_1,\ldots, M_n$ has no \dy-step, then $M$ is obtained from $X_8$ or $Y_8$ by gluing a wheel onto an allowable set $X'\subseteq X$ for some $X\in \{S,C\}$, so it is clear that, up to symmetry, any $4$-element segment $A$ of $M$ is in $\{S,C\}$. Hence $\Delta_A(N)\in \mathcal{P}$. Suppose that $M_1,\ldots, M_n$ has at least one \dy-step. Then it follows from Lemma \ref{tech4seg} that $A$ is equivalent to a $4$-element segment in $\{S,C\}$, so by the definition of a path sequence there is some path sequence that describes $\Delta_A(N)$, and therefore $\Delta_A(N)\in \mathcal{P}$.
\end{proof}

Next we consider path sequences that end in a fan. 

\begin{lemma}
\label{techmaxfans}
Let $M$ be a matroid described by a path sequence $M_1,\ldots, M_n$ ending in a fan $F$. If $F$ is not a maximal fan, then $F\cup \{s\}$ is a maximal fan of $M$ for some $s\in E(M)-A$. Moreover, there is a path sequence $M_1',\ldots, M_t'$ ending in $F\cup \{s\}$ that describes $M$.
\end{lemma}

\begin{proof}
We may assume, by duality, that $M_n$ is obtained from $M_{n-1}$ by gluing a wheel onto an allowable triangle $A'$ of a $4$-element segment $A\in \{S,C\}$. Suppose that $F$ is not a maximal fan. By the path sequence that describes $M$ the fan $F$ is obtained by gluing a wheel onto a triangle $A'$ that is contained in a $4$-element segment $A$, so it follows from orthogonality that no triads meet both $F$ and $E(M)-F$. Thus any fan properly containing $F$ is an extension of $F$ by adding end spoke elements, and the end spoke elements must be parallel with elements of $A'$ in $M_{n-1}$.  But from the definition of a \dy-step we see that $\cl_{M_{n-1}}(A)$ meets $E(M_{n-1})-A$ in at most a rank one set. Thus $F\cup s$ is maximal for some end spoke element $s\in E(M)-F$. It remains to show that there is a path sequence that describes $M$ ending in $F\cup s$.

Since $M_{n-1}$ is described by a path sequence, there must be at least one \dy-step in $M_1,\ldots, M_{n-1}$. 

If there are at least two \dy-steps in the path sequence $M_1,\ldots, M_{n-1}$, then $M_{n-1}\del s$ is described by a path sequence $M_1',\ldots, M_{n-1}'$ ending in $A$ by Lemma \ref{parallelendelements}. Then $M$ is described by a path sequence $M_1',\ldots, M_{n-1}', M_n'$, where $M_n'$ is obtained from $M_{n-1}'$ by gluing a wheel onto $A'$, such that the element of $A'$ that was parallel to $s$ in $M_{n-1}$ is not deleted.

Assume that $M_1,\ldots, M_{n-1}$ has exactly one \dy-step. Then, since $M_{n-1}$ has an element in parallel with $A$, either $M_{n-1}$ is a single-element parallel extension of $Y_8$, where an element of $C$ is in parallel with an element of $S$, or $M_{n-1}$ is obtained from that matroid by gluing a wheel onto a triangle of $S\cup C-A$. In either case, $M_{n-1}\del s$ is described by a path sequence such that $A\in \{S,C\}$, and $M$ can be obtained from $M_{n-1}\del s$ by gluing a wheel onto $A'$, where the element of $A'$ that was parallel to $s$ in $M_{n-1}$ is not deleted.
\end{proof}


The following is another important consequence of Lemma \ref{tech4seg} and \ref{techmaxfans}.

\begin{lemma}
\label{trirecognition}
Let $M\in \mathcal{P}$, and let $A$ be a path-generating allowable triangle of $M$. If $N$ is an allowable parallel extension of $M$ along $A$, then $\Delta_A(N)\in \mathcal{P}$. Moreover, if $A$ is contained in a $4$-element segment of $M$, then $\Delta_A(M)\in \mathcal{P}$.
\end{lemma}

\begin{proof}
We may assume that $M$ is a matroid described by a path sequence $M_1, \ldots, M_n$. Then, since $A$ is a path-generating allowable triangle of $M$, either $A$ is contained in a $4$-element segment or $A$ is contained in a path-generating fan of $M$.

Suppose that $A$ is contained in a $4$-element segment of $M$. If $M_1, \ldots, M_n$ has no \dy-step, then $M$ is obtained from $X_8$ or $Y_8$ by gluing a wheel onto an allowable set $X'\subseteq X$ where $X\in \{S,C\}$. In either case we can assume, up to symmetry, that $A\subseteq S$. Hence $\Delta_A(M)\in \mathcal{P}$, and if $N$ is an allowable parallel extension of $M$ along $A$, then $\Delta_A(N)$ can be obtained from $M$ by gluing a wheel onto $A\subseteq S$, so $\Delta_A(N)\in \mathcal{P}$. We may therefore assume that $M_1, \ldots, M_n$ has at least one \dy-step. Then by Lemma \ref{tech4seg} we may assume that $A$ is contained in a $4$-element segment parallel to $S$. Thus $\Delta_A(M)$ and $\Delta_A(N)$ can be obtained from $M$ by gluing a wheel onto $A\subseteq S$, so both $\Delta_A(M)$ and  $\Delta_A(N)$ belong to $\mathcal{P}$.

Assume that $A$ is contained in a path-generating fan $F$ of $M$. By Lemma \ref{commute3} and Lemma \ref{techmaxfans} we can assume that $M_1, \ldots, M_n$ is a path sequence ending in $F$. Now it follows from Lemma \ref{lem:growwheel} that $\Delta_A(N)$ is either a series extension of $M$ or else $\Delta_A(N)$ can be obtained from $M_{n-1}$ by gluing a fan onto an allowable triangle, so in either case $\Delta_A(N)\in \mathcal{P}$.
\end{proof}


Finally, we show that we can remove an element from a path sequence and obtain, essentially, another path sequence.

\begin{lemma}\label{lem:removefrompathseq}
  Let $M$ be described by a path sequence with $|E(M)| \geq 10$. There is an $e \in E(M) - (S \cup C)$ such that at least one of $M\delete e$, $M\contract e$, $\co(M\delete e)$, and $\si(M \contract e)$ is described by a path sequence.
\end{lemma}

\begin{proof}
  Let $M$ be a counterexample with $|E(M)|$ minimal. Let $M_1, \ldots, M_n$ be a path sequence that describes $M$. 
  Suppose that $M_1, \ldots, M_n$ is a path sequence ending in a fan $F$. If $|F|\geq 4$, then the result follows easily, so we may assume that $|F|=3$. Then it is easy to see that $M_{n-1}$ must be a counterexample as well. Hence we may assume $M_1, \ldots, M_n$ is a \dy-sequence.
  
  Suppose $M_1, \ldots, M_n$ has exactly one \dy-step. Up to duality, we can assume $M = \Delta_S^Q(X_8)$, with $|E(Q) - E(M)| \geq 2$. Let $e \in E(Q) - E(M)$. Then $M \delete e = \Delta_S^{Q\delete e}(X_8)$ is again described by a path sequence. Hence we may assume there are at least two \dy-steps. By Corollary \ref{usingeachend2}, we may assume that the last two \dy-steps both use $A \in \{S,C\}$, and that $A$ is a segment of $M$, so $M = \nabla_A^{N_n}(\Delta_A^{N_{n-1}}(M_{n-2}))$. Let $e \in E(N_n) - E(M_{n-1})$. If $|E(N_n) - E(M_{n-1})| \geq 2$, then $M\contract e = \nabla_A^{N_n\contract e}(M_{n-1})$ can be described by a path sequence. Hence $E(N_n) - E(M_{n-1}) = \{e\}$. Now we have
  \begin{align*}
    M \contract e = \nabla_A(M_{n-1}) = \nabla_A(\Delta_A^{N_{n-1}}(M_{n-2})) = N_{n-1}.
  \end{align*}
  But $N_{n-1}$ is an allowable parallel extension of $M_{n-2}$, so $\si(M\contract e) = M_{n-2}$, which can be described by a path sequence, as desired.  
\end{proof}

\section{The setup}
\label{setup}

We start the proof of our main result as follows.

\begin{lemma}
\label{pathseqminors}
Let $M$ be a matroid described by a path sequence. Then $M$ has a minor in $\mathcal{S}=\{M_{8,6}, X_8,Y_8,Y_8^{*}\}$.
\end{lemma}

\begin{proof}
By Lemma \ref{lem:removefrompathseq}, it suffices to verify the lemma for $|E(M)| \leq 9$. We omit this easy finite case check, which includes matroids from Figures \ref{fig:H5_8}, \ref{fig:H5_9}, as well as some matroids that are not simple or not cosimple. See also \cite[Lemma 4.6]{chunfragile}.
\end{proof}

Our strategy is to prove that the class of $\mathbb{U}_2$- or $\mathbb{H}_5$-representable fragile matroids with an $\mathcal{S}$-minor is contained in $\mathcal{P}$. If $M$ is 3-connected and has a proper $M_{8,6}$-minor, but no minor in $\{X_8, Y_8, Y_8^*\}$, then $M$ has a minor isomorphic to one of $M_{9,7}, M_{9,9}, M_{9,17}, M_{9,19}$. This implies we are in case (ii) or (v) of, respectively, Theorems \ref{3.5} and \ref{3.7}. Hence our strategy will yield the main result.

In what follows, let $M$ be a $\mathbb{P}$-representable fragile matroid $M$ with an $\{X_8,Y_8,Y_8^{*}\}$-minor. Suppose that $M$ is not in the class $\mathcal{P}$, and that $M$ is minimum-sized with respect to this property. Then $M$ is $3$-connected because $\mathcal{P}$ is closed under series-parallel extensions. Moreover, the dual $M^{*}$ is also not in $\mathcal{P}$ because $\mathcal{P}$ is closed under duality. By inspecting the small matroids Figure \ref{9elts}, it is easy to see that $M$ must have at least $10$ elements. Thus, by the Splitter Theorem and duality, we can assume there is some element $x$ of $M$ such that $M\del x$ is also a $3$-connected $\mathbb{P}$-representable fragile matroid with an $\{X_8,Y_8,Y_8^{*}\}$-minor. By the assumption that $M$ is minimum-sized with respect to being outside the class $\mathcal{P}$, it follows that $M\del x\in \mathcal{P}$. Thus $M\del x$ is described by a path sequence $M_1,\ldots, M_n$. We show that there are three possibilities for the structure of $x$ in $M$ relative to the path of $3$-separations associated with $M_1,\ldots, M_n$.

Let $(A,B)$ be a $k$-separation of $M\del x$. Recall that $x$ \textit{blocks} $(A,B)$ if neither $(A\cup x, B)$ nor $(A, B\cup x)$ is a $k$-separation of $M$. We use the following characterisation frequently. 

\begin{prop}
\label{blockclosure}
\cite[Proposition 3.5]{geelen2001conn}
 Let $(A,x,B)$ be a partition of $M$. If $(A,B)$ is an exact $k$-separation of $M\del x$, then $x$ blocks $(A,B)$ if and only if $x$ is not a coloop of $M$, $x\notin \cl_M(A)$ and $x\notin \cl_M(B)$.
\end{prop}
 
We can now locate $x$. 

\begin{lemma}
\label{blocksorguts}
Let $M$ and $M\del x$ be $3$-connected $\{U_{2,5}, U_{3,5}\}$-fragile matroids. If $M\del x$ is described by a path sequence with associated path of $3$-separations $\mathbf{P}$, then either:
\begin{enumerate}
 \item[(i)] $x$ is in the guts of some $3$-separation displayed by $\mathbf{P}$; or
 \item[(ii)] $x$ blocks some $3$-separation displayed by $\mathbf{P}$; or 
 \item[(iii)] for each $3$-separation $(R,G)$ of $M$ displayed by $\mathbf{P}$, there is some $X\in \{R,G\}$ such that $x\in \cl_M(X)$ and $x\in \cl_M^{*}(X)$.
\end{enumerate}
\end{lemma}

\begin{proof}
Assume that neither (i) nor (ii) holds, and let $(R,G)$ be a $3$-separation of $M\del x$ displayed by a maximal path-generating set. Then, by Proposition \ref{blockclosure}, we may assume that $x\in \cl_M(R)$. Now $x\notin \cl_M(G)$ because (i) does not hold, so $x\in \cl_M^{*}(R)$ by Lemma \ref{clandco}. Thus (iii) holds.
\end{proof}

We consider the three possible cases of Lemma \ref{blocksorguts} in the next three sections respectively. In case (i) we will obtain a direct contradiction to the assumption that $M\notin \mathcal{P}$. In cases (ii) and (iii) we bound the size of $M$ so that $M$ is at most a $12$-element matroid. This reduces the proof to a finite case-analysis, from which the contradiction to the assumption that $M\notin \mathcal{P}$ is then obtained.

Note that by Lemma \ref{pathseqminors} we only need to keep the path sequence to keep the $\mathcal{S}$-minor.

\section{Guts case}
\label{guts}

In this section, with $M$ and $x\in E(M)$ as in Section \ref{setup}, we assume that $x$ is in the guts of some $3$-separation displayed by the path of $3$-separations $\mathbf{P}$ associated with a path sequence that describes $M\del x$. 

First we show that $x$ is not in the closure of a $4$-element segment.

\begin{lemma}
\label{cantbesegmentguts}
 If $M\del x$ has a $4$-element segment $S$, then $x\notin \cl_M(S)$.
\end{lemma}

\begin{proof}
Suppose that $M\del x$ has a $4$-element segment $S$, and that $x\in \cl_M(S)$. Then $M$ has a $U_{2,5}$-restriction $S\cup \{x\}$, so $y\in E(M)-\cl_M(S\cup x)$ is flexible, a contradiction because $M$ is fragile.
\end{proof}


Next we show that $M\del x$ cannot be described by a path sequence ending in a $4$-element segment or cosegment. 

\begin{lemma}
\label{nodypath}
$M\del x$ is not described by a path sequence ending in either a $4$-element segment or cosegment. 
\end{lemma}

\begin{proof} 
Suppose that $M\del x$ is described by a path sequence $M_1,\ldots, M_n$ ending in a $4$-element cosegment $A$. Then $A$ is an allowable $4$-element cosegment of $M$ by Lemma \ref{whencsegsgoup}, so $\nabla_A(M)$ is well-defined. Moreover, it follows from Lemma \ref{2.16} that $\nabla_A(M)\del x=\nabla_A(M\del x)$, which is an allowable parallel extension of $M_{n-1}$. Now, by the definition of a \dy-step, there is at least one element $a'\notin A\cup x$ such that $a'$ is parallel with an element of $A$ in $\nabla_A(M)$. Let $M'=\nabla_A(M)\del a'$. Then $|E(M')|<|E(M)|$, so $M'\in \mathcal{P}$ by the minimality of $M$. Now, by Lemma \ref{2.11}, an allowable parallel-extension of $M'$ along $A$ by $a'$ followed by a $4$-element generalized $\Delta$-$Y$ exchange gives a matroid isomorphic to $M$. But $M'\in \mathcal{P}$ and $A$ is a path-generating allowable $4$-element segment of $M'$, so it follows from Lemma \ref{4segrecognition} that $M\in \mathcal{P}$, a contradiction to the assumption that $M\notin \mathcal{P}$. The argument when $M\del x$ is described by a path sequence ending in a $4$-element segment is symmetric.    
\end{proof}

Next we deal with the case where $M\del x$ is described by a path sequence ending in a fan.

\begin{lemma}
\label{nobigfansforguts}
$M\del x$ is not described by a path sequence ending in a fan with at least $4$ elements. 
\end{lemma}

\begin{proof}
Suppose that $M\del x$ is described by a path sequence $M_1,\ldots, M_n$ ending in a fan $F$. The fan $F$ must contain a path-generating allowable triad $T$ of $M\del x$ because $F$ has at least $4$ elements. Since $x\in \cl_M(E(M\del x)-F)$, it follows from Lemma \ref{asegsgoup} and Lemma \ref{whencsegsgoup} that $F$ is a fan of $M$ and that $T$ is a path-generating allowable triad of $M$. Moreover, it follows from Lemma \ref{lem:growwheel} that $\nabla_A(M)$ has at least one element $f\in F-T$ in parallel with an element of $T$. Let $M'=\nabla_T(M)\del f$. Then $|E(M')|<|E(M)|$, so $M'\in \mathcal{P}$ by the minimality of $M$. Now, by Lemma \ref{2.11}, an allowable parallel-extension of $M'$ along $A$ by $f$ followed by a $3$-element $\Delta$-$Y$-exchange gives a matroid isomorphic to $M$. But $M'\in \mathcal{P}$ and $A$ is a path-generating allowable triangle of $M'$, so it follows from Lemma \ref{trirecognition} that $M\in \mathcal{P}$, a contradiction to the assumption that $M\notin \mathcal{P}$. 
\end{proof}


We can now finish off the first case from Lemma \ref{blocksorguts}. 

\begin{lemma}
\label{xnotinguts}
The element $x$ is not in the guts of a $3$-separation displayed by the path of $3$-separations associated with a path sequence that describes $M\del x$.
\end{lemma}

\begin{proof}
Seeking a contradiction, suppose that $x$ is in the guts of a $3$-separation of $M\del x$. Then it follows from Lemma \ref{nodypath} and Lemma \ref{nobigfansforguts} that $M\del x$ is described by a path sequence ending in a triad or triangle. Moreover, since $|E(M\del x)|\neq 8$, some step of the path sequence that describes $M\del x$ must be a \dy-step. 

Suppose first that $M\del x$ is described by a path sequence $M_1,\ldots, M_n$ ending in a triad. We may assume, by Corollary \ref{usingeachend2} and Lemma \ref{commute3}, that the $(n-1)$-th step uses $A'\in \{S,C\}$, that $M_{n-1}$ has a $4$-element segment $A'$, and that $M_n=\Delta_A(M_{n-1})$ for some allowable triangle $A\subseteq A'$ of $M_{n-1}$. Now since $x$ is a guts element of the path generated by $A$, it follows from Lemma \ref{whencsegsgoup} that $A$ is an allowable triad of $M$, so the matroid $\nabla_A(M)$ is well-defined. Moreover, it follows from Lemma \ref{2.11} and Lemma \ref{2.16} that $\nabla_A(M)\del x=\nabla_A(M\del x)=M_{n-1}$, so $A$ is contained in an allowable $4$-element segment $A'$ of $\nabla_A(M)$ by Lemma \ref{asegsgoup}. Now if $\nabla_A(M)\in \mathcal{P}$, then it follows from Lemma \ref{2.11} and Lemma \ref{trirecognition} that $M\in \mathcal{P}$, a contradiction. Therefore $\nabla_A(M)$ is also a minimum-sized counterexample, and $x$ is in the guts of the path of $3$-separations associated with the path sequence $M_1,\ldots, M_{n-1}$ that describes $\nabla_A(M)\del x=M_{n-1}$. But $M_{n-1}$ is described by a path sequence $M_1,\ldots, M_{n-1}$ ending in a $4$-element segment $A'$, which contradicts Lemma \ref{nodypath}. 

We may therefore assume that $M\del x$ can only be described by a path sequence $M_1,\ldots, M_n$ ending in a triangle. We may assume by Corollary \ref{usingeachend2} and Lemma \ref{commute3} that $M_1,\ldots, M_{n-1}$ is a path sequence ending in a $4$-element cosegment $A'$, and that $M_n=\nabla_A(M_{n-1})$ for some allowable triad $A\subseteq A'$ of $M_{n-1}$. Suppose that $x\in \cl_M(A)$. Then $A\cup x$ is a path-generating allowable $4$-element segment of $M$, so $\Delta_{A\cup x}(M)$ is a well-defined fragile matroid. Moreover, there is an element $a'\in A'-A$ in series with an element of $A\cup x$ by the dual of Lemma \ref{cantbesegmentguts}. Since $M$ is minimum-sized outside of $\mathcal{P}$ it follows that $\Delta_{A\cup x}(M)/a'\in \mathcal{P}$. But then $M\in \mathcal{P}$ by Lemma \ref{2.11} and the dual of Lemma \ref{4segrecognition}; a contradiction. Now assume that $x\notin \cl_M(A)$. Now $A$ is an allowable triangle of $M$ by Lemma \ref{asegsgoup}, so $\Delta_A(M)$ is a well-defined fragile matroid and $A$ is contained in a $4$-element cosegment $A'$ of $\Delta_A(M)$. If $\Delta_A(M)\in \mathcal{P}$, then it follows from Lemma \ref{2.11} and the dual of Lemma \ref{trirecognition} that $M\in \mathcal{P}$; a contradiction. Thus $\Delta_A(M)$ is a minimum-sized counterexample and $x$ is in the guts of the path of $3$-separations associated with the path sequence $M_1,\ldots, M_{n-1}$ that describes $\Delta_A(M)\del x=\Delta_A(M\del x)=M_{n-1}$. But $M_{n-1}$ is described by a path sequence $M_1,\ldots, M_{n-1}$ ending in a $4$-element cosegment; a contradiction to Lemma \ref{nodypath}.       
\end{proof}

\section{Blocking case}

In this section, we assume that the matroid $M$ satisfies the conclusion of Lemma \ref{blocksorguts} (ii). That is, $M$ has an element $x$ of $M$ such that: 
\begin{enumerate}
 \item[(B1)] $M\del x$ is $3$-connected and $M\del x \in \mathcal{P}$; and
 \item[(B2)] $x$ blocks some $3$-separation displayed by the path of $3$-separations $\mathbf{P}$ associated with a path sequence that describes $M\del x$. 
\end{enumerate}

Our strategy in this section is to show that a minimum-sized matroid $M$ with respect to properties (B1) and (B2) has at most $11$ elements. The fact that $M$ is minimum-sized gives the following straightforward restriction on the elements of $E(M)-\{x\}$.

\begin{lemma}
\label{forbiddenelements}
There is no element $y\in E(M)-\{x\}$ such that:
\begin{enumerate}
 \item[(i)] $y$ is deletable in $M\del x$, $M\del x,y$ is $3$-connected, $M\del x,y\in \mathcal{P}$, and $x$ blocks some $3$-separation displayed by the path of $3$-separations $\mathbf{P}'$ associated with a path sequence that describes $M\del x,y$; or
 \item[(ii)] $y$ is contractible in $M\del x$, $M\del x/y\in \mathcal{P}$, and $x$ blocks some $3$-separation displayed by the path of $3$-separations $\mathbf{P}'$ associated with a path sequence that describes $\si(M\del x/y)$. 
\end{enumerate}
\end{lemma}

We make the following observations about rank that we will use for finding elements that can be contracted while preserving the property (B2).

\begin{lemma}
\label{outofspan}
Let $M$ be a matroid in $\mathcal{P}$, and $A\subseteq E(M)$.
\begin{enumerate}
 \item[(i)] If $A$ is an allowable $4$-element cosegment of $M$, and $R\subseteq E(M)-A$ is a $3$-separating set in $M$, then $r(R\cup\{a,a'\})=r(R)+2$ for all distinct pairs $a,a'\in A$.
 
 \item[(ii)] If $A=(a_1,\ldots, a_n)$ is a fan with $n\geq 5$ such that the ends of $A$ are rim elements, and $R\subseteq E(M)-A$ is a $3$-separating set in $M$, then $r(R\cup \{a_1, a_n\})=r(R)+2$. 
 
 \item[(iii)] If $M$ is described by a path sequence ending in $A$, and $(R,G)$ is a $3$-separation displayed by the associated path of $3$-separations with $A\subseteq R$, and $y,z\in G$ are contractible elements in coguts steps of the path of $3$-separations, then $r(R\cup \{y,z\})=r(R)+2$.
\end{enumerate}
\end{lemma}
 
We first show that $M\del x$ is not described by a path sequence ending in a large fan.
\begin{lemma}
\label{nobigfans}
$M\del x$ is not described by a path sequence ending in a fan with at least $5$ elements. Moreover, if $M\del x$ is described by a path sequence $M_1,\ldots, M_n$ ending in a fan with $4$ elements, then $M_n$ is obtained from $M_{n-1}$ by gluing a wheel onto an allowable triad of $M_{n-1}$.
\end{lemma}

\begin{proof}
Suppose that $M\del x$ is described by a path sequence ending with a fan $F$ of $M\del x$ such that $|F|\geq 5$. By Lemma \ref{techmaxfans} or its dual we may assume that $F$ is a maximal fan. Suppose that $s\in F$ is a spoke end of $F$. Then $s$ is deletable in $M\del x$, $M\del x,s$ is $3$-connected, $M\del x,s$ is described by a path sequence ending in the fan $F-s$ with $|F-s|\geq 4$. Since $x$ blocks some $3$-separation of $M\del x$ displayed by the path of $3$-separations associated with $F$, it follows that $x$ blocks a $3$-separation of $M\del x,s$ displayed by the path of $3$-separations associated with $F-s$. Then the element $s$ contradicts Lemma \ref{forbiddenelements} (i). We may therefore assume that the ends of $F$ are rim elements. Then it follows from Lemma \ref{outofspan} (ii) that at least one of the ends of $F$, say $r$, satisfies $r\notin \cl_M((E(M)-F)\cup x)$. Then it is easy to see that $r$ is an element that contradicts Lemma \ref{forbiddenelements} (ii). 

For the second statement, suppose that $M\del x$ is described by a path sequence $M_1,\ldots, M_n$ such that $A\in \{S,C\}$ is a $4$-element segment of $M_{n-1}$, and that $M_n$ is obtained from $M_{n-1}$ by gluing a wheel onto an allowable triangle $A'\subseteq A$ of $M_{n-1}$. Let $F$ be the corresponding $4$-element fan. By Lemma \ref{techmaxfans} we may assume that $F$ is a maximal fan. Let $s\in F$ be the end spoke element of $F$. Then $M\del x,s$ is $3$-connected, and $M\del x,s$ is described by the path sequence $M_1,\ldots, M_{n-1}, M\del x,s$ ending in the triad $F-s$. Since $x$ blocked a $3$-separation of $M\del x$ displayed by the path of $3$-separations associated with $M_1,\ldots, M_n$ it follows that $x$ also blocks a $3$-separation of $M\del x,s$ displayed by the path of $3$-separations associated with $M_1,\ldots, M_{n-1}, M\del x,s$. But then $s$ contradicts Lemma \ref{forbiddenelements} (i). 
\end{proof}

Finding elements to remove such that the path sequence structure is preserved is important in what follows. We now prove some technical results that will enable us to find such elements. The first observation is on $4$-element segments and, by duality, on $4$-element cosegments.

\begin{lemma}
\label{Pm4segs}
If $M'$ is described by a path sequence $M_1,\ldots, M_n$ ending in a $4$-element segment $A$, then there is some $a\in A$ such that $M'\del a$ is $3$-connected and $M_1,\ldots, M_{n-1}, M'\del a$ is a path sequence where $M'\del a$ is obtained from $M_{n-1}$ by gluing a wheel onto $A-a$. Moreover, if $M'=\nabla_A^Q(M_{n-1})$ for some $Q$ such that $|E(Q)-E(M_{n-2})|\geq 2$, then there are at least two such elements in $A$.
\end{lemma}

\begin{proof}
Since $M_1,\ldots, M_n$ is a path sequence ending in $A$ we may assume that $M_n=\nabla_{A}^{Q}(M_{n-1})$ for some $Q$. Let $a\in A$ be an element that is in series with an element $q\in E(Q)-E(M_{n-1})$ in $Q$. Then $M'\del a$ is $3$-connected since $A$ is a $4$-element segment of $M'$, and $M'\del a=\nabla_{A-a}^{Q/q}(M_{n-1})$ by Lemma \ref{2.13}. Thus $M'\del a$ is obtained from $M_{n-1}$ by gluing a wheel onto the allowable triad $A-a$ of $M_{n-1}$. 
\end{proof}

We next consider path sequences ending in a triad or triangle.

\begin{lemma}
 \label{Pmtris}
Let $M_1,\ldots, M_{n-1}$ be a path sequence ending in a $4$-element cosegment $A$, and let $M'=\Delta_{A'}(M_{n-1})$ for some allowable triad $A'\subseteq A$ of $M_{n-1}$. If $M'$ is $3$-connected, and $M'$ cannot be described by a path sequence where $A'$ is contained in a fan with at least $4$-elements, then either:
\begin{enumerate}
 \item[(i)] $M'\del a'$ is $3$-connected for some $a'\in A'$, and $M_1,\ldots, M_{n-2}, M'\del a'$ is a path sequence, where $M'\del a'$ is obtained from $M_{n-2}$ by gluing a wheel onto $A-a'$; or
 \item[(ii)] $M'/a$ is $3$-connected for $a\in A-A'$, and $M_1,\ldots, M_{n-2}$ is a path sequence that describes $M'/a$.
\end{enumerate}
\end{lemma}

\begin{proof}
Since $A'$ is not contained in a fan with at least $4$-elements, it follows that $M'\del a'$ is $3$-connected for each deletable element $a'\in A'$. Moreover it follows from Lemma \ref{2.13} and Lemma \ref{2.16} that $M'\del a'=M_{n-1}/a'=\Delta_{A-a'}(Q\del a')$, where $Q$ is the matroid such that $M_{n-1}=\Delta_A^Q(M_{n-2})$. Therefore (i) holds if $a'$ is in parallel with an element of $E(Q)-E(M_{n-2})$ in $Q$.

Now suppose that $M'\del a'\notin \mathcal{P}$ for all deletable elements $a'\in A'$. Since $M'\del a'=M_{n-1}/a'=\Delta_{A-a'}(Q\del a')$ it follows that $Q\del a'$ is not isomorphic to an allowable extension of $M_{n-2}$ for each deletable $a'\in A'$. But $Q$ must add at least one element in parallel with a deletable element of $A$ by the definition of a \dy-step, so there is a single element of $E(Q)-E(M_{n-2})$ and that element is in parallel with the element $a\in A-A'$. Then it follows from Lemma \ref{2.11} and Lemma \ref{2.13} that $M'/a=Q\del a$. But $Q\del a$ is isomorphic to $M_{n-2}$ because $a$ was the only element in parallel with an element of $E(Q)-E(M_{n-2})$. Therefore (ii) holds.
\end{proof}

Now we eliminate $4$-element segments. 

\begin{lemma}
 \label{nolongmiddle}
If $|E(M\del x)|\geq 10$, then $M\del x$ does not have a path-generating allowable $4$-element segment.
\end{lemma}

\begin{proof}
Seeking a contradiction, suppose that $M\del x$ has a path-generating allowable $4$-element segment $A$. Let $M_1,\ldots, M_n$ be a path sequence that describes $M\del x$. Since $|E(M\del x)|\geq 10$ and, by Lemma \ref{nobigfans}, $M\del x$ is not described by a path sequence ending in a fan with at least $5$ elements, it follows that $M_1,\ldots, M_n$ has at least one \dy-step. Thus $A\in \{S,C\}$ by Lemma \ref{tech4seg}. Hence, by Corollary \ref{usingeachend2} and Lemma \ref{commute3}, we can assume that $M_1,\ldots, M_n$ is a path sequence ending in $A$. But then it follows from Lemma \ref{Pm4segs} that $M\del x$ has an element that contradicts Lemma \ref{forbiddenelements} (i).
\end{proof}

We can now prove the main result of this section:

\begin{lemma}
\label{bigblocking}
$|E(M\del x)|\leq 10$.
\end{lemma}

\begin{proof}
Let $M_1,\ldots, M_n$ be a path sequence that describes $M\del x$. We may assume by Lemma \ref{nobigfans} and Lemma \ref{nolongmiddle} that $M\del x$ cannot be described by a path sequence ending in a either a $4$-element segment or a fan with at least $5$ elements. Seeking a contradiction, suppose that $|E(M\del x)|\geq 11$. Clearly $M_1,\ldots, M_n$ must have at least one \dy-step.

\begin{claim}
\label{atleast2}
$M_1,\ldots, M_n$ has at least two \dy-steps.
\end{claim}

\begin{subproof}
We consider the cases for the end step of $M_1,\ldots, M_n$.

First suppose that $M_1,\ldots, M_n$ is a path sequence ending in a $4$-element cosegment $A$. Then $M_n=\Delta_A^{N}(M_{n-1})$ for some $N$. Seeking a contradiction, suppose that $M_1,\ldots, M_n$ has exactly one \dy-step. Now $|E(N)-E(M_{n-1})|\geq 2$ because $|E(M\del x)|\geq 11$. Then it follows from the dual of Lemma \ref{Pm4segs} that there are at least two elements $a,a'\in A$ such that $M\del x/a$ and $M\del x/a'$ are $3$-connected and $M\del x/a, M\del x/a'\in \mathcal{P}$, and $x$ blocks a $3$-separation of at least one of $M\del x/a$ and $M\del x/a'$ by Lemma \ref{outofspan} (i); a contradiction to Lemma \ref{forbiddenelements} (ii). 

Now suppose that $M\del x$ is described by a path sequence $M_1,\ldots, M_n$ such that $A$ a $4$-element cosegment of $M_{n-1}$, and that $M_n$ is obtained from $M_{n-1}$ by gluing a wheel onto an allowable triad $A'\subseteq A$ of $M_{n-1}$. Let $F$ be the corresponding $4$-element fan. Assume that $M\del x$ is described by a path sequence with exactly one \dy-step and let $A=C$. By Corollary \ref{usingeachend2} and Lemma \ref{commute3}, and the fact that $M\del x$ cannot be described by a path sequence ending in either a $4$-element segment or cosegment, it follows that $M\del x$ can also be described by a path sequence whose end step is a fan $F'$ corresponding to a wheel glued onto an allowable triad of $S$. That is, $M_1,M_2, M_3$ is a path sequence that describes $M\del x$, where $M_2$ is obtained from $X_8$ by gluing a wheel onto $A'$, $M_3=\Delta_S^N(M_2)$, and $M_4$ is obtained from $M_3$ by gluing a wheel onto an allowable subset of $S$. Then $|E(N)-E(M_2)|\geq 2$ because $|E(M\del x)|\geq 11$. Suppose that $F'$ is a $4$-element fan, and let $s\in F'$ be the end spoke element. Then $M\del x,s$ is $3$-connected and has a path sequence $M_1,M_2, M_3$. Hence by Lemma \ref{commute3} $M\del x,s$ can be described by a path sequence ending in $F$. But $x$ must still block a $3$-separation of $M\del s$; a contradiction to Lemma \ref{forbiddenelements} (i). Thus we may assume that $F'$ is a triangle of $M\del x$. Then it follows from Lemma \ref{Pmtris} and the fact that $|E(N)-E(M_2)|\geq 2$ that there is some $s\in F'$ such that $M\del x,s$ is $3$-connected and $M\del x,s$ is described by a path sequence $M_1,M_2,M\del x,s$. But $M\del x,s$ has a path sequence ending in $F$ by Lemma \ref{commute3}, and $x$ must still block a $3$-separation of $M\del s$; a contradiction to Lemma \ref{forbiddenelements} (i).

Suppose that $M\del x$ is described by a path sequence $M_1,\ldots, M_n$ ending in a triangle. By Corollary \ref{usingeachend2} and Lemma \ref{commute3}, we can assume that $M\del x$ is described by a path sequence $M_1,\ldots, M_n$ such that $A$ a $4$-element segment of $M_{n-2}$, $M_{n-1}=\Delta_A^N(M_{n-2})$, and finally $M_n=\nabla_{A'}(M_{n-1})$ for some allowable triad $A'\subseteq A$ of $M_{n-1}$. Moreover, $|E(N)-E(M_{n-2})|\geq 2$ because $|E(M\del x)|\geq 11$. Then it follows from Lemma \ref{Pmtris} that, for some $a'\in A'$, $M\del x, a'$ is $3$-connected and $M\del x, a'$ is described by a path sequence ending $M_1,\ldots, M_{n-1}, M\del x, a'$ ending in a fan $F$ spanned by $A-a'$. But $x$ must block a $3$-separation of $M\del x,a'$; a contradiction to Lemma \ref{forbiddenelements} (i).
 
Finally, suppose that $M_1\ldots M_n$ is a path sequence ending in a triad. Then, by Corollary \ref{usingeachend2} and Lemma \ref{commute3}, we can assume that $A$ is a $4$-element cosegment of $M_{n-2}$, that $M_{n-1}=\nabla_A^N(M_{n-2})$, and finally that $M_n=\Delta_{A'}(M_{n-1})$ for some allowable triangle $A'\subseteq A$ of $M_{n-1}$. Moreover, we can assume that $M\del x$ is not described by a path sequence ending in a fan $F$ with $|F|\geq 4$, a $4$-element segment or cosegment, or triangle. Thus, if $M_1\ldots M_n$ has exactly one \dy-step, then $x$ must block the $3$-separation $(A', E(M\del x)-A')$. Furthermore $|E(N)-E(M_{n-2})|=3$ because $|E(M\del x)|\geq 11$. Then $M\del x$ has an element $y\in E(N)-E(M_{n-2})$ that contradicts Lemma \ref{forbiddenelements} (ii), since at most one element from $E(N)-E(M_{n-2})$ can be spanned by $A'\cup x$ by Lemma \ref{outofspan} (iii). Thus there are at least two \dy-steps in the path sequence $M_1,\ldots, M_n$.
\end{subproof}

By \ref{atleast2} the path sequence $M_1\ldots M_n$ has at least two \dy-steps. Thus, by Corollary \ref{usingeachend2} and Lemma \ref{commute3}, we may assume that the path sequence $M_1\ldots M_n$ has least two \dy-steps on $A$. We can assume by Corollary \ref{usingeachend2} that $A=S$.

We first assume that $M_n$ is obtained from $M_{n-1}$ by gluing a $4$-element fan onto an allowable triad $A'\subseteq A$ of $M_{n-1}$. Let $a\in A-A'$, and let $G=E(M\del x)-F$. Suppose that $x\notin \cl_M(F\cup a)$. Then $x$ must block some $3$-separation $(X,Y)$ of $M\del x$ where $F\cup a\subseteq X$. Let $s\in F$ be the end spoke element of the fan $F$. Then $M\del x,s$ is $3$-connected and can be described by a path of $3$-separations $M_1,\ldots, M_{n-1}$ ending in the $4$-element cosegment $(F-s)\cup a$. But then $(X-s,Y)$ is a $3$-separation of $M\del x,s$ displayed by the path of $3$-separations associated with $M_1,\ldots, M_{n-1}$, and $x$ blocks $(X-s,Y)$. Thus $s$ contradicts Lemma \ref{forbiddenelements} (i). We may therefore assume that $x\in \cl_M(F\cup a)$. Then $x$ must block the $3$-separation $(F, G)$ of $M\del x$. Since $M_1,\ldots, M_n$ has at least two \dy-steps on $A$ it follows that there is an internal coguts element $r\in G$. But then, up to parallel elements, $M\del x/r$ is described by a path sequence ending in $F$ by Lemma \ref{Pmintcoguts}, and $x$ blocks $(F,G-r)$ in $M/r$ because $r\notin \cl_M(F\cup a)$ and hence $x\notin \cl_M(F\cup r)$. Thus $r$ contradicts Lemma \ref{forbiddenelements} (ii).   

It remains to consider the cases where $M_1,\ldots, M_{n-1}$ is a path sequence ending in a $4$-element segment or cosegment $A$, and $M_1,\ldots, M_{n}$ is either a path sequence ending in a $4$-element segment $A$ or a path sequence ending in a triangle or triad $A'\subseteq A$. Let $B=E(M\del x)-A$.

\begin{claim}
$x$ blocks the $3$-separation $(A,B)$.  
\end{claim}

\begin{subproof}
We show that $x\notin \cl_M(A)$ and $x\notin \cl_M(B)$, so the claim will follow from Proposition \ref{blockclosure}.

Since $x$ blocks some $3$-separation $(R,G)$ of $M\del x$ displayed by the path of $3$-separations generated by the end step $A$ or $A'\subseteq A$ of $M_1,\ldots, M_n$, it follows from Proposition \ref{blockclosure} that $x\notin \cl_M(A)$. 

Assume that $M_1,\ldots, M_n$ is a path sequence ending in a triad $A'\subseteq A$. Suppose that $x\in \cl_M(B)$. Then since $x$ blocks a $3$-separation $(R,G)$ displayed by the path of $3$-separations associated with $M_1,\ldots, M_n$ where $A\subseteq R$ and $G\subseteq B$. Since $x\in \cl_M(B)$ the triad $A'$ of $M\del x$ is also an allowable triad of $M$ by Lemma \ref{whencsegsgoup}, so $\nabla_{A'}(M)$ is fragile. Moreover, by Lemma \ref{2.16}, we have $\nabla_{A'}(M)\del x=\nabla_{A'}(M\del x)=M_{n-1}$. Now $x$ must block a $3$-separation $(R,G)$ displayed by the path of $3$-separations associated with $M_1,\ldots, M_{n-1}$. But $M_1,\ldots, M_{n-1}$ is a path sequence ending in a $4$-element segment $A$, which contradicts Lemma \ref{nolongmiddle}.

Suppose that $M_1,\ldots, M_n$ is a path sequence ending in a $4$-element cosegment $A$, then it follows from the dual of Lemma \ref{Pm4segs} that there is some $a\in A$ such that $M\del x/a$ is described by a path sequence ending in a fan $F$ that contains $A-a$. Thus by Lemma \ref{forbiddenelements} (ii) $x$ does not block a $3$-separation of $M\del x/a$ displayed by the path of $3$-separations generated by $F$. That is, $x$ does not block a $3$-separation of the form $(R-a,G)$ in $M/a$ where $A\subseteq R$ and $G\subseteq B$. Thus $x\in \cl_M(G\cup a)-\cl_M(G)$ by Lemma \ref{blockclosure}, and $a\in \cl_M(G\cup x)\subseteq \cl_M(B\cup x)$ by the Mac Lane-Steinitz exchange property. But $a\notin \cl_M(B)$, so it follows that $\cl_M(B)\neq \cl_M(B\cup x)$. Hence $x\notin \cl_M(B)$.

Finally, suppose that $M_1,\ldots, M_n$ is a path sequence ending in a triangle $A'\subseteq A$. Suppose that $x\in \cl_M(B)$. Then $x$ blocks a $3$-separation $(R,G)$ displayed by the path of $3$-separations associated with $M_1,\ldots, M_n$ where $A\subseteq R$ and $G\subseteq B$. By Lemma \ref{asegsgoup} the set $A'$ is an allowable triangle of $M$, and $\Delta_{A'}(M)\del x=\Delta_{A'}(M\del x)=M_{n-1}$. But $x$ must still block a $3$-separation $(R,G)$ displayed by the path of $3$-separations associated with $M_1,\ldots, M_{n-1}$. But $M_1,\ldots, M_{n-1}$ is a path sequence ending in a $4$-element cosegment $A$, so now by the above argument $x\notin \cl_{M_{n-1}}(B)=\cl_M(B)$, which contradicts the assumption that $x\in \cl_M(B)$.  
\end{subproof}

Since $M_1,\ldots, M_n$ has at least two \dy-steps it follows that there is an internal coguts element $r\in B$.  Then it follows from Lemma \ref{Pmintcoguts} that, up to parallel elements, $M\del x/r$ is described by a path sequence ending in the same allowable set as $M\del x$. By Lemma \ref{forbiddenelements} (ii) $x$ does not block a $3$-separation of $M\del x/r$, so in particular $x$ does not block $(A,B-r)$ in $M/r$. Hence $x\in \cl_M(A\cup r)-\cl_M(A)$ by Lemma \ref{blockclosure}. It follows from the Mac Lane-Steinitz exchange property that $r\in \cl_M(A\cup x)$.

\begin{claim}
\label{rands}
There is some $c\in B-r$ such that $M\del x/c$ can be described by a path sequence ending in the same allowable set as $M_1,\ldots, M_n$. 
\end{claim}

\begin{subproof}
By Corollary \ref{usingeachend2} and Lemma \ref{commute3} we can assume there is a path sequence $M_1',\ldots, M_n'$ describing $M\del x$ ending in a $4$-element cosegment $C\subseteq B$ or a triangle or triad $C'\subseteq C\subseteq B$. If $C$ is a $4$-element cosegment, then it follows from the dual of Lemma \ref{Pm4segs} that $M\del x/c$ is described by a path sequence ending in a fan $F$ containing $C-c$ for some $c\in C$. But again by Corollary \ref{usingeachend2} and Lemma \ref{commute3} $M\del x/c$ can be described by a path sequence ending in either the cosegment $A$ or the triad or triangle $A'\subseteq A$.  

Suppose that $M_1',\ldots, M_n'$ is a path sequence ending in a triangle or triad $C'\subseteq C\subseteq B$. We may assume by Corollary \ref{usingeachend2} and Lemma \ref{commute3} that $M_1',\ldots, M_{n-1}'$ is a path sequence ending in $C$. By Lemma \ref{forbiddenelements} (i) and Lemma \ref{Pmtris} or its dual, it follows that $M\del x/c$ is $3$-connected and described by a path sequence $M_1',\ldots, M_{n-2}'$ or $M_1',\ldots, M_{n-2}', M\del x/c$. Again by Corollary \ref{usingeachend2} and Lemma \ref{commute3} it follows that $M\del x/c$ can be described by a path sequence ending in either the cosegment $A$ or the triad or triangle $A'\subseteq A$.
\end{subproof}

It follows from Lemma \ref{forbiddenelements} (ii) that $x$ does not block the $3$-separation $(A,B-c)$ of $M/c$. Then $x\in \cl_M(A\cup c)-\cl_M(A)$ by Lemma \ref{blockclosure}, so $c\in \cl_M(A\cup x)$ by the Mac Lane-Steinitz exchange property. But now we have $c,r\in \cl_M(A\cup x)$, which contradicts Lemma \ref{outofspan} (iii).     
\end{proof}

\section{Spanning and cospanning case}

In this section, we assume that the matroid $M$ satisfies the conclusion of Lemma \ref{blocksorguts} (iii). That is, $M$ has an element $x$ of $M$ such that: 
\begin{enumerate}
 \item[(i)] $M\del x$ is $3$-connected and $M\del x \in \mathcal{P}$; and
 \item[(ii)] For each $3$-separation $(R,G)$ displayed by the path of $3$-separations $\mathbf{P}$ associated with a path sequence that describes $M\del x$, there is some $X \in \{R,G\}$ such that $x \in \cl_M(X)$ and $x\in \cl_M^*(X)$. 
\end{enumerate}

\begin{lemma}
\label{spanningend}
If $M\del x$ is described by a path sequence $M_1,\ldots, M_n$ ending in a $4$-element segment or cosegment $A$, then $\cl^{*}_M(A)$ spans and cospans $x$ in $M$.
\end{lemma}

\begin{proof}
Assume that $M_1,\ldots, M_n$ is a path sequence for $M\del x$ ending in a $4$-element cosegment $A$. We may assume that $M_n=\Delta_A^N(M_{n-1})$ for some $N$. Seeking a contradiction, suppose that $x$ is not spanned and cospanned by $A$. Then $A$ is an allowable $4$-element cosegment of $M$ by Lemma \ref{whencsegsgoup}, so $\nabla_A(M)$ is fragile. Moreover, $\nabla_A(M)\del x=N$ by the dual of Lemma \ref{2.16} (ii) and Lemma \ref{2.11}. Now there is some $y\in E(N)-E(M_{n-1})$ in parallel with an element of $A$ in $N$ by the definition of a \dy-step. Let $M'=\nabla_A(M)\del y$. Then $M'\in \mathcal{P}$ by the minimality of $M$, and $M'$ has a path-generating allowable $4$-element segment $A$. But now $M=\Delta_A^{N'}(M')$ for some $N'$. Hence $M\in \mathcal{P}$ by Lemma \ref{4segrecognition}; a contradiction.

Assume that $M\del x$ is described by a path sequence $M_1,\ldots, M_n$ ending in a $4$-element segment $A$. Let $M_n=\nabla_A^Q(M_{n-1})$, and let $B=E(M\del x)-\cl^{*}(A)$. Suppose that $x\in \cl_M(B)$ and $x\in \cl_M^{*}(B)$. Then $A$ is an allowable $4$-element segment of $M$ by Lemma \ref{asegsgoup}, so $\Delta_A(M)$ is fragile. Now $\Delta_A(M)\del x=\Delta_A(M\del x)=Q$ by Lemma \ref{2.16}. Thus $\Delta_A(M)\del x$ has at least one element $q\in Q$ in series with some element of $A$, and since $x\in \cl_M^{*}(B)$ it follows that $\Delta_A(M)$ has some element of $Q$ in series with some element of $A$. By the minimality of $M$ it follows that $\Delta_A(M)/q\in \mathcal{P}$. But then it follows from the dual of Lemma \ref{4segrecognition} that $M\in \mathcal{P}$, which is a contradiction.
\end{proof}

The next result follows using essentially the same argument as Lemma \ref{nobigfansforguts}.

\begin{lemma}
\label{fanspancospan}
If $M\del x$ is described by a path sequence ending in a fan $F$ with $|F|\geq 4$, then $x\in \cl_M(F)$ and $x\in \cl_M^{*}(F)$.
\end{lemma}


Next we consider path sequences ending in a triangle or triad.

\begin{lemma}
\label{trispan}
Let $M_1,\ldots, M_n$ be a path sequence that describes $M\del x$.
\begin{enumerate}
\item[(a)] If $M_1,\ldots, M_n$ is a path sequence ending in a triad $A$, and $x$ is spanned and cospanned by $E(M\del x)-A$, then $\nabla_A(M)$ is a minimum-sized counterexample such that $\nabla_A(M)\del x$ is described by the path sequence $M_1,\ldots, M_{n-1}$ and satisfies Lemma \ref{blocksorguts} (iii).

\item[(b)]  If $M_1,\ldots, M_n$ is a path sequence ending in a triangle $A$, and $x$ is spanned and cospanned by $E(M\del x)-\cl^{*}(A)$, then $\Delta_A(M)$ is a minimum-sized counterexample such that $\Delta_A(M)\del x$ is described by the path sequence $M_1,\ldots, M_{n-1}$ and satisfies Lemma \ref{blocksorguts} (iii).
\end{enumerate}
\end{lemma}

\begin{proof}
We prove (a), with the proof of (b) being essentially the same argument. Suppose that that $M\del x$ can be described by a path sequence $M_1,\ldots, M_n$ ending in a triad $A$, where $M_n$ is obtained from $M_{n-1}$ by gluing a wheel onto an allowable triangle of $S$. Since $x$ is spanned and cospanned by $E(M\del x)-A$ it follows from Lemma \ref{whencsegsgoup} that $A$ is an allowable triad of $M$. Hence $\nabla_A(M)$ is fragile. If $\nabla_A(M)\in \mathcal{P}$, then so is $M$ by Lemma \ref{2.11} and Lemma \ref{4segrecognition}; a contradiction. Thus $\nabla_A(M)\notin \mathcal{P}$, so $\nabla_A(M)$ is also a minimum-sized counterexample. Moreover, we see that $\nabla_A(M)\del x=M_{n-1}$ by Lemma \ref{2.16}, so $\nabla_A(M)\del x$ is described by the path sequence $M_1,\ldots, M_{n-1}$. Since Lemma \ref{blocksorguts} (iii) holds for $M\del x$, and every $3$-separation displayed by the path of $3$-separations associated with the path sequence $M_1,\ldots, M_{n-1}$ is also displayed by the path of $3$-separations associated with the path sequence $M_1,\ldots, M_n$, it follows that Lemma \ref{blocksorguts} (iii) holds for $\nabla_A(M)\del x$.
\end{proof}

As a consequence of Lemma \ref{trispan} we can assume that $M$ is chosen so that if $M\del x$ is described by a path sequence ending in a triangle or triad $A$, then $x$ is spanned and cospanned by $\cl^{*}(A)$.

We can now show that a bound can be obtained unless $M\del x$ is described by a path sequence ending in a fan.

\begin{lemma}
 \label{bighasfan}
If $M\del x$ is described by a path sequence $M_1,\ldots, M_n$ ending in a fan $F$ with $|F|\geq 4$, then $M\del x$ is obtained from $X_8$ by gluing a wheel onto an allowable subset. Moreover, if $M\del x$ cannot be described by a path sequence ending in a fan $F$ with $|F|\geq 4$, then $|E(M\del x)|\leq 11$.
\end{lemma}

\begin{proof}
Suppose that $M\del x$ is described by a path sequence $M_1,\ldots, M_n$ ending in a fan $F$ with $|F|\geq 4$. Then $x$ is spanned and cospanned by $F$ by Lemma \ref{fanspancospan}. Assume that $M_n$ is obtained from $M_{n-1}$ by gluing a wheel onto an allowable subset of $S$. Seeking a contradiction, suppose that $M\del x$ is described by a path sequence $M_1',\ldots, M_n'$ ending in either a $4$-element segment or cosegment $C$ or a fan $F'$, where $M_n'$ is obtained from $M_{n-1}'$ by gluing a wheel onto an allowable subset of $C$. Suppose $M_1',\ldots, M_n'$ is a path sequence ending in either a $4$-element segment or cosegment $C$. Then, since $x$ is spanned and cospanned by $F$, it follows that $x$ cannot be spanned and cospanned by $\cl^{*}(C)$; a contradiction to Lemma \ref{spanningend}. Suppose that $M_1',\ldots, M_n'$ is a path sequence ending in a fan $F'$, where $M_n'$ is obtained from $M_{n-1}'$ by gluing a wheel onto an allowable subset of $C$. Then, since $x$ is spanned and cospanned by $F$, it follows from the choice of $M$ that $|F'|\geq 4$. But $x$ cannot be spanned and cospanned by $F'$; a contradiction to Lemma \ref{fanspancospan}. Therefore the path sequence $M_1,\ldots, M_n$ has no \dy-steps.

For the second statement, suppose that $M\del x$ cannot be described by a path sequence ending in a fan $F$ with $|F|\geq 4$, but that $|E(M\del x)|\geq 12$. Then it follows that $M\del x$ is described by a path sequence $M_1,\ldots, M_n$ with at least two \dy-steps. By Lemma \ref{commute3} and Corollary \ref{usingeachend2} we can assume that $M_1,\ldots, M_n$ is a path sequence ending in either a $4$-element segment or cosegment $S$, or a triangle or triad $T$ where $M_n$ is obtained from $M_{n-1}$ by gluing a wheel onto an allowable subset of $S$. In the former case it follows from Lemma \ref{spanningend} that $x$ is spanned and cospanned by $\cl^{*}(S)$, while in the latter case it follows from the choice of $M$ that $x$ is spanned and cospanned by $\cl^{*}(T)$. But, by Lemma \ref{commute3} or Corollary \ref{usingeachend2} together with the choice of $M$, it follows that $M\del x$ is also described by a path sequence $M_1',\ldots, M_n'$ is a path sequence ending in either a $4$-element segment or cosegment $C$. Since $x$ is spanned and cospanned by either $\cl^{*}(S)$ or $\cl^{*}(T)$, and $M_1',\ldots, M_n'$ has at least two \dy-steps, it follows that $x$ cannot be spanned and cospanned by $\cl^{*}(C)$; a contradiction to Lemma \ref{spanningend}. Therefore $M_1,\ldots, M_n$ has at most one \dy-step, and hence $|E(M\del x)|\leq 11$.
\end{proof}

Let $F$ be a fan of a matroid $M'\del x$ with fan ordering $(f_1,\ldots, f_n)$. We say that $x$ is on a \textit{guts line} of $F$ if there are fan elements $f_i$ and $f_j$ of $F$ with $i<j$ such that $x$ is in the guts of the $3$-separation $(R,G)$ of $M'$, where $R=\{f_i, f_{i+1},\ldots, f_{j-1}, f_j\}$.

\begin{lemma}
 \label{fangutsline}
Let $M_1,\ldots, M_n$ be a path sequence ending in a fan $F$ with $|F|\geq 4$ that describes $M\del x$. Then $x$ is not on a guts line of $F$.
\end{lemma}

\begin{proof}
Seeking a contradiction, suppose that $x$ is on a guts line of the fan $F$. 

\begin{claim}
\label{dyredfanseg2}
For each triangle $\{f_i,f_{i+1},f_{i+2}\}$ of $M\del x$ contained in $F$, there is no $4$-element segment $A=\{x,f_i,f_{i+1},f_{i+2}\}$ of $M$.
\end{claim}

\begin{subproof}
Suppose that $A=\{x,f_i,f_{i+1},f_{i+2}\}$ is a $4$-element segment of $M$ for some triangle $\{f_i,f_{i+1},f_{i+2}\}$ of $M\del x$ contained in $F$. Since the rim element $f_{i+1}$ is contractible in $M\del x$, it follows that $f_{i+1}$ is contractible in $M$. Thus $f_{i+1}$ is non-deletable in $M$, so it follows that $A$ is a path-generating allowable $4$-element segment of $M$. Thus $\Delta_A(M)$ is fragile and $A$ is an allowable $4$-element cosegment of $\Delta_A(M)$ with non-contractible element $f_{i+1}$. Let $S=\cl_{M}^{*}(A)-A$. Since $F$ is a fan of $M\del x$ with $|F|\geq 4$, it follows that $S$ is non-empty. Moreover, the members of $S$ are in series with contractible elements of $A$ in $\Delta_A(M)$ because $A\cup S$ cannot contain a $5$-element cosegment by the dual of Lemma \ref{cantbesegmentguts}. Let $M'=\Delta_A(M)/S$. Then $M'\in \mathcal{P}$ because $M$ is minimum-sized with respect to not being in $\mathcal{P}$. But $M'$ has a path-generating allowable $4$-element cosegment $A$, so it follows from the dual of Lemma \ref{4segrecognition} and Lemma \ref{2.11} that $M\in \mathcal{P}$; a contradiction.
\end{subproof}

We can assume, by \ref{dyredfanseg2}, that $x$ is not spanned by a triangle contained in $F$. Suppose that $T=\{x, f_1,f_2\}$ is a triangle of $M$, where $f_1$ a rim element of $F$. Then $T$ is an allowable triangle of $M$, and $\Delta_T(M)$ has an allowable $4$-element cosegment $A=T\cup \{f_3\}$. But then it follows from \ref{dyredfanseg2} that $(\Delta_T(M))^{*}\in \mathcal{P}$, hence $M\in \mathcal{P}$ by Lemma \ref{trirecognition}; a contradiction.

We may therefore assume that $x$ is in the guts of the $3$-separation $(R,G)$ of $M\del x$, where $R=\{f_i, f_{i+1},\ldots, f_{j-1}, f_j\}$ for non-consecutive $i,j\in \{1,\ldots, n\}$. Since $x\in \cl_M(G)$ it follows from orthogonality that $x\notin \cl_M^{*}(R)$, so an allowable triad $T=\{f_k,f_{f+1},f_{k+2}\}$ of $M\del x$ that is contained in $R$ is an allowable triad of $M$. It follows from Lemma \ref{lem:growwheel} that there is some non-empty set $S\subseteq \{f_{k-1},f_{k+3}\}$ of elements of $\nabla_T(M)\del x$ in parallel with elements of $T$. Now $\nabla_T(M)\del S\in \mathcal{P}$ because $M$ is minimum-sized with respect to not being in $\mathcal{P}$, and $T$ is a path-generating allowable triangle of $\nabla_T(M)\del S$. Hence $M\in \mathcal{P}$ by Lemma \ref{trirecognition}; a contradiction.
\end{proof}

We now bound the size of $F$ when $M\del x$ is described by a path sequence ending in a fan $F$. 

\begin{lemma}
\label{smallspanningfans}
If $M\del x$ is described by a path sequence ending in a fan $F$, then $|F|\leq 5$ and $x$ blocks $(T, E(M\del x)-T)$ for some triad $T$ of $F$.
\end{lemma}

\begin{proof}
Suppose that $M\del x$ is described by a path sequence ending in a fan $F$. We can assume that $F$ is a maximal fan of $M$ by Lemma \ref{techmaxfans}, and that $x\in \cl_M(F)$ and $x\in \cl_M^{*}(F)$ by Lemma \ref{fanspancospan}.

\begin{claim}
\label{blocking}
$x$ blocks a $3$-separation $(R,G)$ of $M\del x$, where $R=\{f_i,f_{i+1},\ldots, f_j\}$ for some $i<j$.
\end{claim}

\begin{subproof}
Seeking a contradiction, suppose that $x$ does not block any $3$-separation of $M\del x$ nested with $F$. Let $T=\{f_{i-1}, f_i, f_{i+1}\}$ be a triad of $F$. Suppose $x\in \cl_M(T)$. Since $M\del x$ is obtained from $X_8$ by gluing a wheel onto an allowable subset of $S$ or $C$, we can obtain an $M_{7,0}$ minor with triad $T$ by contracting the rim elements of $F$ outside of the span of $T$, then contracting elements of $C$ until $T$ is spanning, and then simplifying the resulting matroid. Removing the same elements of $M$ gives a fragile extension of $M_{7,0}$ by $x$, and the only such matroid is $Y_8$ where $\{f_{i-1},f_i,x\}$ or $\{f_i,f_{i+1},x\}$ is a triangle. Therefore either $\{f_{i-1},f_i,x\}$ or $\{f_i,f_{i+1},x\}$ is a triangle of $M$. But then $x$ is on a guts line of $F$; a contradiction to Lemma \ref{fangutsline}. Thus $x\notin \cl_M(T)$. Since we assume that $x$ cannot block $(T, E(M\del x)-T)$, it follows that $x\in \cl(E(M\del x)-T)$ by Lemma \ref{blockclosure}. Hence $\nabla_T(M)$ is a fragile matroid, and $\nabla_T(M)\del x=\nabla_T(M\del x)$. By Lemma \ref{lem:growwheel} the matroid $\nabla_T(M)\del x$ has a fan $F'=F-P$ where $P\subseteq \{f_{i-1}, f_{i+1}\}$ is some non-empty set of elements in parallel with elements of $F'$. Let $M'=\nabla_T(M)\del P$ and let $T'$ be the triangle of $M'$ containing $f_i$. Then we see that $x\in \cl_{M'}(F')$, and that for any subset $F''$ such that $T'\subseteq \cl_{M'}(F'')$, $x\in \cl_{M'}(F'')$ if and only if $x\in \cl_{M}(F''\cup T)$, while for subsets $F''$ such that $T'$ is not contained in $\cl_{M'}(F'')$, $x\in \cl_{M'}(F'')$ if and only if $x\in \cl_{M}(F'')$. Thus $x$ does not block a $3$-separation of $M'\del x$ nested with $F'$. But we can repeat this process until $F'$ has a single triad $T''$, so that $x\in \cl_{M'}(T'')$. Then $x$ is on a guts line of $F$; a contradiction to Lemma \ref{fangutsline}.   
\end{subproof}

Now by \ref{blocking} we may assume that $x$ blocks the $3$-separation $(R,G)$. We use the minimality of $M$ to show that $|F|\leq 5$. If $|F|\geq 5$ and $F$ has a spoke end $s$, then $M\del x,s$ is described by a path sequence ending in the fan $F-s$ and $x$ blocks a $3$-separation nested with $F-s$. Thus if $|F|\geq 5$ then we can assume that the ends of $F$ are rim elements. If $R$ contains at least two triads, then it follows from Lemma \ref{outofspan} (ii) that there is some rim element $r\in R$ such that $\si(M\del x/r)$ is described by a path sequence ending in a fan $F'$ contained in $F-r$ and $x$ blocks a $3$-separation nested with $F'$. Thus we can assume $R$ contains a single triad. Similarly, if $F\cap G$ has at least two rim elements, then there is some rim element $g\in F\cap G$ such that $\si(M\del x/g)$ is described by a path sequence ending in a fan $F'$ contained in $F-g$ and $x$ blocks a $3$-separation nested with $F'$. Therefore $|F|\leq 5$, and $x$ blocks $(T, E(M\del x)-T)$ for some triad $T$ of $F$.
\end{proof}

Thus a largest minimum-sized counterexample $M$ has 11 elements. 

\begin{lemma}
\label{9elts}
$|E(M\del x)|\leq 11$.
\end{lemma}

\begin{proof}
We may assume by Lemma \ref{bighasfan} that $M\del x$ is described by a path sequence $M_1,M_2$ ending in a fan $F$ with $|F|\geq 4$. Then $|F|\leq 5$ by Lemma \ref{smallspanningfans}. Hence $|E(M\del x)|\leq 10$.
\end{proof}

\section{Wrapping up}

Finally we can prove our main result.

\begin{proof}[Proof of Theorem \ref{thm:mainresult}]
  Suppose $M$ is a minimal counterexample to the theorem. Up to duality, let $x \in E(M)$ be such that $M\delete x$ is 3-connected and has a $U_{2,5}-$ or $U_{3,5}$-minor. By Lemma \ref{blocksorguts}, there are three possibilities for $x$. By Lemmas \ref{xnotinguts}, the first case cannot occur. By Lemma \ref{bigblocking} and Lemma \ref{9elts}, a minimal counterexample has at most 12 elements. The result now follows from a finite case check, which we verified by computer in \cite[Lemma 4.5]{chunfragile}.
\end{proof}

\bibliographystyle{acm}
\bibliography{fr}
\end{document}